\documentclass[reqno]{amsart}

\usepackage{amsmath}
\usepackage{amssymb}
\usepackage{amsthm}
 \usepackage{paralist, mathrsfs, eucal}

\usepackage{a4wide}
\usepackage{bm}

\usepackage{amsmath,amsfonts,amsthm,amssymb,amsxtra}
\usepackage{amsxtra, amssymb, mathrsfs, pstricks}

\newtheorem{theorem}{Theorem}[section]
\newtheorem{lemma}[theorem]{Lemma}
\newtheorem{corollary}[theorem]{Corollary}
\newtheorem{proposition}[theorem]{Proposition}

\newtheorem{assum}[theorem]{Assumption}

\theoremstyle{definition}
\newtheorem{definition}[theorem]{Definition}
\newtheorem{example}[theorem]{Example}

\theoremstyle{remark}
\newtheorem{remark}[theorem]{Remark}
\theoremstyle{notation}
\newtheorem{notation}[theorem]{Notation}

\numberwithin{equation}{section}
\usepackage{xcolor}
\usepackage{paralist}

\definecolor{ddcyan}{rgb}{0,0.1,0.9}
\definecolor{ddmagenta}{rgb}{0.8,0,0.8}
\definecolor{orange}{rgb}{0.6,0.2,0}

\def\Rr{\mathbb{R}}     
\def\Nn{\mathbb{N}}     
\def\Mm{\mathcal{M}}    
\def\Fc{\mathcal{F}}
\def\Ic{\mathcal{I}}

\def\BV{BV}

\def\Xs{X^*}

\def\Var{\text{\normalfont{Var}}}   

\def\graph{\text{\normalfont{graph}}}
\def\scal#1{\left\langle #1 \right\rangle}
\def\scalx#1{\left\langle  #1  \right\rangle}

\def\Ec{\mathcal{E}}            
\def\Ic{\mathcal{I}}

\def\veps{\varepsilon}
\def\dd{\,{\rm d}}
\def\Ll{\mathcal{L}}    

\def\dom{\text{\normalfont{dom}}}
\def\supp{\text{\normalfont{supp}}}

\def\wconv{\rightharpoonup}
\newcommand{\weakto}{\rightharpoonup}
\newcommand{\weaksto}{\wstar}
\def\wstar{\stackrel{*}{\wconv}}

\newcommand{\Mosco}{\stackrel{\text{\rm \tiny M}}{\longrightarrow}}
\newcommand{\Graph}{\stackrel{\text{\rm \tiny g}}{\longrightarrow}}

\def\Gammaliminf{\mathop{\Gamma\text{--}\mathrm{liminf}}}
\def\Gammalimsup{\mathop{\Gamma\text{--}\mathrm{limsup}}}

\newcommand{\fitz}{\varphi}
\newcommand{\prodpi}[2]{\pi(#2,#1)}
\newcommand{\foraa}{\text{for a.a.}}
\newcommand{\R}{\mathbb{R}}
\newcommand{\N}{\mathbb{N}}
\newcommand{\epsi}{\varepsilon}

\def\trait #1 #2 #3 {\vrule width #1pt height #2pt depth #3pt}
\def\fin{
    \trait .3 5 0
    \trait 5 .3 0
    \kern-5pt
    \trait 5 5 -4.7
    \trait 0.3 5 0
\medskip}

\newcommand{\ene}[2]{\mathcal{E}_{#1}(#2)}
\newcommand{\cg}{\mathcal{G}}
\newcommand{\frsubname}{\partial \cE}
\newcommand{\frsub}[2]{\frsubname_{#1}(#2)}
\newcommand{\diff}[2]{\frsubname_{#1}(#2)}
\newcommand{\cE}{\mathcal{E}}
\newcommand{\la}{\langle}
\newcommand{\ra}{\rangle}
\newcommand{\eps}{\varepsilon}
\newcommand{\sing}{\mathrm{sin}}
\newcommand{\ac}{\mathrm{ac}}
\newcommand{\uai}{\mathcal{Y}}
\newcommand{\Rzd}{\R^{3\times 3}_{\rm dev}}

\def\ricky#1{{\color{ddmagenta}R: #1}}

\newenvironment{rickynew}{\color{red}}{\color{black}}
\newcommand{\berin}{\begin{rickynew}}
\newcommand{\erin}{\end{rickynew}}

\title[Stability results for doubly-nonlinear equations]{Stability results for doubly nonlinear differential
inclusions by variational convergence}

\author{Thomas Roche}
\address{Thomas Roche\\ Department of Mathematics / M6 \\ Technische Universit\"at M\"unchen \\ Boltzmannstr. 3 \\ 85748 Garching b. M\"unchen \\ Germany}
\email{roche@ma.tum.de}

\author{Riccarda Rossi}
\address{Riccarda Rossi\\ Dipartimento di Matematica\\ Universit\`{a} di Brescia\\ Via Valotti 9\\ I-25133 Brescia\\ Italy}
\email{riccarda.rossi@ing.unibs.it}

\author{Ulisse Stefanelli}
\address{Ulisse Stefanelli\\ Istituto di Matematica Applicata e Tecnologie Informatiche {\it E. Magenes} - CNR\\ Via Ferrata 1\\ I-27100 Pavia\\ Italy}
\email{ulisse.stefanelli@imati.cnr.it}

\date{}

\thanks{T.R. would like to thank IMATI Pavia and the University of Brescia, where part of this work was conducted, for the kind hospitality. T.R. was supported by the Bavarian Network of Excellence through its graduate program 'TopMath' and the TUM Graduate School through its Thematic Graduate Center 'TopMath'}

\begin{document}

 \begin{abstract}
We present a stability result for  a wide class  doubly nonlinear equations,
 fea\-tu\-ring  general maximal monotone operators, and (possibly) nonconvex and nonsmooth energy functionals.  The limit analysis resides on the reformulation of the differential evolution as a scalar energy-conservation equation with the aid of the so-called Fitzpatrick theory
 for the  re\-pre\-sen\-ta\-tion   of monotone operators. 
 In particular, our result applies to the vanishing viscosity
 approximation of
 rate-independent systems.
 \end{abstract}



 \maketitle

\section{\bf Introduction}
This note is concerned with a convergence result for  doubly nonlinear differential inclusions of the type
\begin{equation}
 \alpha_n\left(\dot u_n(t) \right) + \partial \Ec_t(u_n(t)) \ni 0 \quad \text{in } X^* \qquad \foraa\, t \in (0,T).\label{eq:DNE-approx}
\end{equation}
Here,  $(\alpha_n)$
is a sequence of maximal monotone  (and possibly multivalued) operators $\alpha_n:X\rightrightarrows\Xs$, $(X,\| \cdot \|)$ is a (separable) reflexive Banach space, and  $\cE: [0,T] \times X \to (-\infty,\infty] $ is a  (proper)
time-dependent {\it energy} functional. We will prove that for $\alpha_n\rightarrow \alpha$ in the graph sense any limit point $u$ of the sequence $(u_n)_{n\in\Nn}$ is a solution to
\begin{equation*}
 \alpha\left(\dot u(t) \right) + \partial \Ec_t(u(t)) \ni 0 \quad \text{in } X^* \qquad \foraa\, t \in (0,T)
\end{equation*}

Throughout the paper, we write $\ene
tu$ in place of $\cE(t,u)$. We will understand
 the multivalued operator
$\frsubname : (0,T) \times X \rightrightarrows X^*$  to be  a
suitable notion of subdifferential for the possibly nonsmooth  and
nonconvex map $u \mapsto \ene tu$, namely the so-called the \emph{Fr\'echet subdifferential},
defined at $(t,u)\in \mathrm{dom}(\cE)$ by
\begin{equation}
\label{frsub-def} \xi \in \frsub tu \quad \text{ if and only if }
\ene tv \geq \ene tu +\scalx{\xi,v-u}+ {\rm o}(\|v-u\|) \quad \text{as $v
\to u$.}
\end{equation}
 Observe that, as soon as the mapping $u \mapsto \ene tu$ is \emph{convex},
the Fr\'echet subdifferential $\diff tu$ coincides with the subdifferential
of $u \mapsto \ene tu$ in the sense of convex analysis.

Doubly nonlinear equations as in \eqref{eq:DNE-approx} arise in a variety of different applications, ranging from Thermomechanics, to phase change, to magnetism. As such, they have attracted a substantial deal of attention in recent years. Correspondingly, the related literature is quite rich. Being beyond our scope to attempt here a comprehensive review, we limit ourselves  to
 recording the seminal observations   by  {\sc Moreau} \cite{Moreau70,Moreau71} and  {\sc Germain} \cite{Germain73},  as well as the early existence results by {\sc Arai} \cite{Arai79}, {\sc Senba} \cite{Senba86}, {\sc Colli \& Visintin} \cite{Colli-Vis90}, and {\sc Colli} \cite{Colli92ter}. The reader can find a selection of more recent results
 in
  \cite{Akagi-Otani04,Akagi08,Akagi11,Akagi11c,Emmrich11,KA12,MRS11x,RMS08,sss}.
 Without going into details, let us mention that, over the last decade, the convexity requirement on the map  $u \mapsto \ene tu$
 in  the pioneering papers \cite{Arai79,Senba86,Colli-Vis90,Colli92ter} has been progressively weakened: in particular,
 in   \cite{MRS11x} a quite broad class of nonsmooth and nonconvex energy functionals has been considered.
 Nonetheless,  in all of the aforementioned contributions the operator $\alpha$ is assumed to fulfill some coercivity property, namely  to have
 at least linear  growth at infinity.  We will refer to this case as {\it viscous}.

The case of $0$-homogeneous operators $\alpha$ has been recently investigated as well, for it connects with the modeling of so-called rate-independent systems. We shall hence refer to this situation as {\it rate-independent}. Some references in this direction are to be found in the papers \cite{DalMaso05,DalMaso04,Francfort-Mielke06,Mainik-Mielke05,Mielke05,Mielke04,Mielke-Roubicek05,Mielke-Theil04,Mielke-et-al02}.

Additionally, relation \eqref{eq:DNE-approx} has been considered in connection with the study of the long-time behavior of solutions in \cite{Akagi11e,Efendief-Zelik09,ss08,segatti2006,sss}, and their variational characterization in \cite{Akagi10,Akagi11d,be,be2}.

The focus of this paper is on the study of the stability of the
doubly nonlinear flows \eqref{eq:DNE-approx}. Namely, we investigate
the convergence of solutions $u_n$ to  equations
\eqref{eq:DNE-approx}, under the assumption  
\begin{equation}
\label{graph-convergence}
 \alpha_n\rightarrow \alpha \text{ in the graph sense   in } X \times \Xs,
 \end{equation}
 viz.\ for all $(\xi,\xi^*) \in \mathrm{graph}(\alpha)$ there
 exists $(\xi_n,\xi_n^*) \in \mathrm{graph}(\alpha_n)$ such that $\xi_n \to \xi$
  in $X$
 and $\xi_n^*\to \xi^*$ in $X^*$   as
 $n \to \infty$.
The main result of this paper, Theorem \ref{thm:main-ref},
 states that  cluster points $u$ of the curves $(u_n)$ are in fact solutions to the limiting equation
 \begin{equation}
 \alpha\left(\dot u(t) \right) + \partial \Ec_t(u(t)) \ni 0 \quad \text{in } X^* \qquad \foraa\, t \in (0,T).\label{eq:DNE-limiting-intro}
\end{equation}

We have to mention that stability results for the doubly nonlinear flows \eqref{eq:DNE-approx} are already available in the literature. For {\it viscous} graphs $\alpha_n$,
a first convergence    theorem  in the case of  convex energies  has been obtained by {\sc Aizicovici \& Yan} \cite{Aizicovici-Yan} (see also \cite{be}), whereas stability results for doubly nonlinear equations with nonconvex energies
have been proved in  \cite{MRS11x} .
 This issue has been  recently reconsidered by {\sc Visintin} \cite{Visintin11,Visintin13,Visintin-press}, who
has
remarkably extended the reach of the theory to treat
subdifferential inclusions of the type
\[
\beta(\dot u(t) ) + \gamma(u(t)) \ni 0 \quad \text{in } X^* \qquad \foraa\, t \in (0,T),
\]
with $\beta, \, \gamma : X \rightrightarrows X^*$ maximal monotone operators,
$\beta$ cyclically monotone and
$\gamma$  noncyclic monotone,
 by
 resorting to
 the so-called {\sc Fitzpatrick} theory  \cite{Fitzpatrick88}.

 Let us briefly
  recall that an operator  $\alpha : X \rightrightarrows X^* $ is {\it cyclically monotone} if $\alpha$ is the {\it generalized gradient} of some potential. Namely, if $\alpha = \partial \psi$
 for some proper, convex, and lower semicontinuous function $\psi: X \to (-\infty,\infty]$,
  where the symbol $\partial$ here denotes the subdifferential in the sense of convex analysis. In the cyclic-monotone case
   $\alpha = \partial \psi$, it is well known that the relation $y \in \partial \psi (x)$ can be equivalently reformulated as $\scalx{ y,x} = \psi(x)+\psi^*(y)$, where $\psi^*$ is the Legendre-Fenchel conjugate of $\psi$ and $\scal{\cdot,\cdot}$ is the duality pairing between $X^*$ and $X$.
   The use of this variational  fact   for the aim of   variationally reformulating   evolution equations dates back to {\sc Brezis-Ekeland} \cite{Brezis-Ekeland76,Brezis-Ekeland76b} and {\sc Nayroles} \cite{Nayroles76,Nayroles76b}. Among the many contributions stemming
    from  this idea,  the reader is especially referred to the existence proofs by {\sc Auchmuty} \cite{Auchmuty93} and {\sc Roubicek} \cite{Roubicek00}, and to the recent monograph by {\sc Ghoussoub} \cite{Ghoussoub08} on self-dual variational principles (see also the references in \cite{be}).

The Fitzpatrick theory allows us to extend this variational view
to subdifferential inclusions of the type \eqref{eq:DNE-limiting-intro}, with
$\alpha$ possibly \emph{noncyclic} monotone,
by introducing  {\it representative} functions   $f_\alpha: X \times X^* \to (-\infty,\infty]$   for the  
operator $\alpha$. These are convex functions $f_\alpha$ with the property 
\begin{equation}
\label{fitzpatrick-intro}
\begin{aligned}
  \forall (x,y) \in X \times  X^*, \ \scalx{ y,x} &\leq f_\alpha(x,y) \ \text{and} \\
y \in \alpha(x) \ \text{iff} \ \scalx{ y,x}&=f_\alpha(x,y).
\end{aligned}
\end{equation}
The reader is referred to Section \ref{s:3} below for a selection of relevant results within this theory.  In particular, in \cite{Visintin11}   these tools
are used in order to reformulate variationally relations \eqref{eq:DNE-approx}  for noncyclic monotone operators. This reformulation opens the way to devise a suitable $\Gamma$-convergence analysis toward structural stability of the flows.

 As for the {\it rate-independent} case, one shall mention the stability results for hysteresis operators from the classical monographs \cite{Brokate-Spre96,Krejci98,Visintin94} (see also \cite{lombardo}). Another stability result in the rate-independent setting is in \cite{be}.
  Moreover, we record {\sc Visintin} \cite{Visintin13,Visintin-press},  which exploits the Fitzpatrick idea in the rate-independent context, but by taking perturbations in $\partial \cE_n$ (again, in a possibly noncyclic monotone frame). 

 Finally, the approximation of rate-independent flows by viscous flows (in the cyclic-monotone case)
 has been recently attracted a great deal of attention.  This is especially critical as viscous and rate-independent evolutions usually call for different analytical treatments. The  \emph{vanishing viscosity} approach to
 abstract rate-independent systems has been in particular developed in
  \cite{Efendiev-Mielke06,Mielke-et-al08,MRS12}. More specifically, in the latter two papers
 it has been shown that the vanishing viscosity limit leads to
   the notion of  $BV$   solution to a rate-independent
   system. In  the recent  \cite{MieRosSav2012}, still within the cyclic-monotone frame, the  $p_n\to 1$ limit, where $p_n$ is the homogeneity of the potential $\psi_n$ of $\alpha_n$, has been addressed,
    and it has been proved that   $BV$   solutions arise in the limit. A stability result  with respect to variational convergence for the latter solution concept has also been obtained.
\paragraph{\bf Our result}
The focus here is that of obtaining a stability result for  the differential inclusions  \eqref{eq:DNE-approx}  by allowing for maximal generality on the perturbations $\alpha_n$ and on the functional $\cE$. In particular, we shall neither assume super-linear equi-coercivity in $\alpha_n$, nor cyclic monotonicity. As for the energy $\cE$, we do not require neither smoothness nor convexity with respect to $u$, but still ask for lower semicontinuity and some coercivity, see Assumption \ref{ener} below.

This generality sets our result aside from the available contributions on this topic.  In particular,
our analysis also encompasses the passage from viscous to rate-independent doubly nonlinear evolution. Indeed,
we are  able to treat here the  $p_n\to 1$ case
 for noncyclic monotone operators $\alpha_n$ (in this setting, $p_n $ is the coercivity exponent for $\alpha_n$). In  our general context,
 we prove that the so-called \emph{local solutions} \cite{Mielke-et-al08,MRS12} to a rate-independent system arise in the $p_n\to 1$ limit.

 The basic idea for handling the noncyclic monotone case, is  to resort to a variational reformulation of the flows \eqref{eq:DNE-approx}
  which is   well suited  for discussing limits. By letting $f_{\alpha_n}$ {\it represent} the monotone operator $\alpha_n$
in the sense of \eqref{fitzpatrick-intro}
 and assuming the validity of a suitable chain rule for the energy $\cE$, relation \eqref{eq:DNE-approx}  is proved to be equivalent (see Proposition \ref{prop:sol}) to
  an  {\it energy conservation} identity,  namely
\begin{equation}
\label{enid-intro}
 \underbrace{\cE_t(u_n(t))}_{\text{\rm \scriptsize  energy at $t$}}  + \underbrace{\int_0^t f_{\alpha_n}\Big(\dot u_n(s), - \partial \cE_s(u(s))\Big)d t}_{\text{\rm \scriptsize dissipated energy on $[0,t]$}} = \underbrace{\cE_0(u(0)}_{\text{\rm initial energy}} + \underbrace{\int_0^t \partial_t \cE_s(u(s)) d s}_{\text{\rm \scriptsize work of ext. actions}}
 \ \text{ for all $t\in [0,T]$.}
\end{equation}
The strategy  is then to prove that, by passing to the $\liminf$ in \eqref{enid-intro}, the structure of the relation is preserved. In particular, we provide sufficient conditions under which the  $\liminf$ of the integral of the representative functions $f_{\alpha_n}$ is a representative function of the limit graph $\alpha$.  Care here is given to developing such a lower semicontinuity argument
for functions which are only
$BV$ in time. This allows us   to directly  include the case of rate-independent flows.

Let us once more  emphasize   that we can encompass in
our analysis a broad class of time-dependent energies $\cE: [0,T]
\times X \to (-\infty,+\infty]$,  (possibly) nonsmooth and nonconvex
with respect to the state variable $u$, but still satisfying a
suitable set of coercivity and regularity type conditions mutuated
from \cite{MRS11x}. Note that in such a general setting existence of
absolutely continuous solutions
 $u_n : [0,T] \to X$  to \eqref{eq:DNE-approx} is presently not known.
A possibility in order to overcome this would be to strengthen our
assumptions on the  energy functional $\cE$, for instance assuming
it to be a  suitable perturbation of a ($\lambda$-)convex
functional, as in  \cite{RS06,RMS08,MRS11x}.
 We however refrain from
this, for the sake of keeping maximal generality for the convergence
result. In this respect, our result should be regarded purely as a
stability  analysis, with  focus on the convergence properties of
the  operators $(\alpha_n)$. A stability result with respect to
suitable convergence of the energy functionals  could also be
obtained, closely following the lines of  \cite[Thm.\ 4.8]{MRS11x}.
Again, we have chosen not to detail this, in order to highlight the
usage of the Fitzpatrick theory to deal with the \emph{noncyclic}
operators $(\alpha_n)$.  This very  generality  will allow us,
for instance, to address in the upcoming Section \ref{ss:4.3-ulisse}
the quasistatic limits of a  class of \emph{dynamical} problems,
which can be in fact reformulated as doubly nonlinear equations of
the form \eqref{eq:DNE-approx}.

\noindent
\textbf{\bf Structure of the paper.} Section \ref{s:3} contains
some background material on Fitzpatrick theory, and on the notions
of variational convergence for functionals and operators which will
be relevant for the subsequent analysis.
 In Section \ref{s:4.1}, some further preliminaries of measure
 theory and convex analysis are provided, whereas in Section
 \ref{s:4.2} the  basic assumptions
on the energy functional $\cE$ are stated in detail, and suitable
reformulations of \eqref{eq:DNE-approx} are discussed. In  Section
\ref{s:4.3-new}  we state our main stability result Theorem
\ref{thm:main-ref} and thoroughly comment it. We also give two
corollaries  (i.e.  Theorem\ \ref{prop:thomas} and Proposition
\ref{lemma:rate-independent}) in two particular cases: specifically,
Prop.\  \ref{lemma:rate-independent} deals with the $p_n \to 1$
vanishing-viscosity limit. We conclude Sec.\ \ref{s:4.3-new} by
discussing classes of energy functionals to which our results apply
 (cf.\ Sec.\ \ref{ss:4.3-nr}), and developing applications to
rate-independent limits of Hamiltonian systems (in Sec.\
\ref{ss:4.3-ulisse}).
The proof of Theorem \ref{thm:main-ref} is
developed throughout Section \ref{s:proof-thm-1},
 also exploiting some  results from Young measure theory
which are contained in
  Appendix \ref{s:appendixB}.

 \section{\bf Fitzpatrick theory}
\label{s:3}
Within this section, 
we shall systematically use the notation
\[
\pi(\xi,\xi^*):= \scalx{ \xi^*, \xi} \quad \text{for all } (\xi,\xi^*) \in X \times X^*
\]
for the duality pairing between the reflexive space $X$ and $X^*$,
and identify possibly multivalued operators $\alpha: X \rightrightarrows X^*$ with the corresponding graphs $\alpha \subset X \times X^* $ without changing notation. We recall that
 $\alpha: X \rightrightarrows X^*$ is \emph{monotone} if
$$ \scalx{\xi^* - \xi_0^*, \xi - \xi_0 } \geq 0 \quad \text{for all } \xi^*\in\alpha(\xi), \  \xi_0^*\in\alpha(\xi_0)$$
(where $\scalx{\cdot, \cdot}$ denotes the duality pairing between
$X^*$ and $X$), and that  it is \emph{maximal monotone,} if it is maximal for set inclusion within the
class of monotone operators.

We shall  provide here a minimal aside on the Fitzpatrick theory,  essentially mutuated from \cite{Visintin08}. The reader is however referred to \cite{Buliga08,Buliga10,Burachik-Svaiter02,Burachik-Svaiter03,Ghoussoub06,Ghoussoub08,Martinez-Legaz-Svaiter05,Martinez-Legaz-Svaiter08,Penot04,Penot04b} for additional material and a collection of related results and applications to PDEs.
\paragraph{\bf Representative functionals.} We denote by $\Fc(X)$ the set of functionals $\fitz: X \times X^* \to   (-\infty,\infty]$ such that
$$\fitz \ \ \text{is convex, lower semicontinuous, and} \ \ \fitz(\xi,\xi^*) \geq \pi( \xi, \xi^*)\quad \forall (\xi,\xi^*) \in X \times X^*.$$
We associate with  $\fitz\in {\Fc}(X)$ the set $\alpha\subset X \times X^*$ given by
 \begin{equation}
 \label{alpha-associated-with}
 (\xi,\xi^*)\in \alpha \quad \Leftrightarrow \quad \fitz(\xi,\xi^*) = \pi(\xi^*,\xi).
 \end{equation}
Whenever \eqref{alpha-associated-with} holds we say that $\fitz$ {\it represents} $\alpha$, that $\fitz$ is {\it representative}, and that $\alpha$ is {\it representable}. A representable operator can be represented by
different representative functionals (cf.\ Example \ref{multiplicity-rmk} below). On the contrary, each representative functional represents only one operator.
\begin{example}
\label{ex:bipotential}
\upshape
In the cyclically monotone case of $\alpha = \partial \psi$ for some
(proper) convex and lower semicontinuous potential $\psi: X \to (-\infty,\infty]$,
  a representative functional for $\alpha$ is given
by the  \emph{bipotential}  (according to the terminology of \cite{Buliga08})
\begin{equation}
\label{example-bipotential}
 \fitz(\xi,\xi^*)  = \psi(\xi)+ \psi^*(\xi^*).
\end{equation}
\end{example}

We have the following {\it strict} set inclusions
\[
\{\text{maximal monotone operators}\} \subsetneqq \{\text{representable operators}\} \subsetneqq \{\text{monotone operators}\}.
\]
Namely, representable operators are intermediate between monotone and maximal monotone. One may wonder how to translate maximality at the level of representative functionals.
The following result provides a useful criterium for the representability of a maximal monotone operator.
\begin{proposition}[Representative of a maximal monotone operator \cite{Sveiter03}]\label{repres}
A functional
$\fitz \in \Fc(X)$ represents a maximal monotone operator iff $\fitz^* \in \Fc(X^*)$. In this case, if $\fitz$ represents $\alpha$ then $\fitz^*$ represents~$\alpha^{-1} = \{ (\xi,\xi^*) \, : \, (\xi^*,\xi )\in \alpha\}$.
\end{proposition}
\paragraph{\bf The Fitzpatrick and the Penot functions.}
Given $\alpha \subset X \times X^*$ with $\alpha \not = \emptyset$ we define the {\it Fitzpatrick function}
(associated with $\alpha$)
$f_\alpha:  X \times X^* \to (-\infty,\infty]$ by
\begin{equation}
\label{fitz-formula}
\begin{aligned}
f_\alpha(\xi,\xi^*)&:= \pi( \xi,\xi^*) + \sup\{ \prodpi{ \xi^* - \xi_0^*}{\xi_0 - \xi} \ : \  (\xi_0,\xi_0^*)\in \alpha\}\\
&= \sup\{ \prodpi{\xi^*}{ \xi_0 } - \prodpi{  \xi_0^*}{ \xi_0 - \xi} \ : \  (\xi_0,\xi_0^*)\in \alpha\} \quad \text{for all } (\xi,\xi^*) \in X \times X^*
\end{aligned}
\end{equation}
and the {\it Penot function} (associated with $\alpha$)  $\rho_\alpha:  X \times X^* \to (-\infty,\infty]$ by
\begin{equation}
\label{penot-formula}
\rho_\alpha := (\pi + I_\alpha)^{**}.
\end{equation}
Both $f_\alpha$ and $\rho_\alpha$ represent $\alpha$. Moreover $f_\alpha$ and $\rho_\alpha$ are respectively the minimal and maximal element (with respect to pointwise ordering) of the functional {\it interval}
\begin{equation}
\label{functional_interval}
{\Ic}(\alpha)=\{ \fitz \in {\Fc}(X) \ : \ \fitz \ \text{represents} \ \alpha\}.
\end{equation}
In particular, in the cyclically monotone case of $\alpha = \partial \psi$, there holds
\begin{equation}
\label{bipot-in-betw}
f_\alpha(\xi,\xi^*) \leq  \psi(\xi)+ \psi^*(\xi^*) \leq \rho_\alpha(\xi,\xi^*) \quad \text{for all } (\xi,\xi^*) \in X \times X^*.
\end{equation}

Let us also point out that, in view of the definitions \eqref{fitz-formula} and
\eqref{penot-formula} of $f_\alpha$ and $\rho_\alpha$ the following formulae hold
\begin{align}
&
\label{1dual-fitz-penot}
f_\alpha(\xi,\xi^*)  = \sup\{   \prodpi{\xi^*}{ \xi_0 } + \prodpi{  \xi_0^*}{\xi}- \rho_{\alpha^{-1}}(\xi_0^*,\xi_0) \, : \,  (\xi_0,\xi_0^*)\in X {\times} X^*\} \quad \forall\,(\xi,\xi^*) \in X {\times} X^*,
\\
&
\label{2dual-fitz-penot}
\rho_\alpha(\xi,\xi^*)  = \sup\{   \prodpi{\xi^*}{ \xi_0 } + \prodpi{  \xi_0^*}{\xi}- f_{\alpha^{-1}}(\xi_0^*,\xi_0) \, : \,  (\xi_0,\xi_0^*)\in X {\times} X^*\} \quad \forall\, (\xi,\xi^*) \in X {\times} X^*,
\end{align}
namely $f_\alpha$ ($\rho_\alpha$, resp.) is the convex conjugate of  the Penot (Fitzpatrick, resp.) function
of the inverse operator $\alpha^{-1}$.

Finally, for later use
we observe that
\begin{equation}
\label{to-be-observed-later}
0 \in \alpha (0)  \ \Rightarrow \ f_\alpha(\xi,\xi^*) \geq 0 \quad \text{for all } (\xi,\xi^*) \in  X \times X^*.
\end{equation}
\paragraph{\bf Self-dual representatives.}
Recall that a function $\fitz: X \times X^* \to (-\infty,\infty]$ is called {\it self-dual} iff
$$\fitz(\xi,\xi^*) = \fitz^*(\xi^*,\xi) \quad \text{for all }(\xi,\xi^*)\in  X \times X^* .$$
The interval ${\Ic}(\alpha)$  from  \eqref{functional_interval} includes {\it self-dual } representative functions \cite[Thm.\ 3.3]{Penot04b}.
In the cyclically monotone case of $\alpha = \partial \psi$, an example in this direction is given
by the  \emph{bipotential} \eqref{example-bipotential}.
Out of the cyclically monotone realm, an example of a self-dual representative in the case $\alpha= \gamma + \partial \psi$ with $\gamma$ skew adjoint is $(\xi, \xi^*)\mapsto \psi(\xi) + \psi^*(-\gamma \xi + \xi^*)$ \cite{Ghoussoub06}.

In the general case, the indirect proof of the existence of self-dual representative functions is due to {\sc Penot} \cite{Penot03,Penot04} and {\sc Svaiter} \cite{Sveiter03},
 whereas direct constructions have been firstly provided by {\sc Penot \& Z\u alinescu} \cite{Penot-Zalinescu05} under some restriction on $\alpha$. An explicit self-dual representative function in the general maximal monotone case has been recently obtained by {\sc Bauschke \& Wang} \cite{Bauschke-Wang08} and reads
\begin{equation}
\label{bau-wang}
(\xi,\xi^*) \mapsto \frac12 \inf_{(\zeta,\zeta^*)\in X \times X^*}\big\{f_\alpha(\xi + \zeta,\xi^* + \zeta^*)+ f_\alpha(\xi - \zeta,\xi^* - \zeta^*)+ \| \zeta\|^2 + \|\zeta \|^2_*\big\}.
\end{equation}
Note that neither the Fitzpatrick function $f_\alpha$ nor the Penot function $\rho_\alpha$ are self-dual in general.

Let us now recast the characterization of maximal monotonicity of Proposition \ref{repres} in the following.
\begin{proposition}[Self-dual representatives = maximality]\label{repres2}
 An
 operator $\alpha: X \rightrightarrows X^*$ is maximal monotone iff it is represented by a self-dual functional $\fitz$.
\end{proposition}
\begin{proof}
 By \cite{Bauschke-Wang08},  if $\alpha$ is maximal monotone, then it admits the
self-dual representative \eqref{bau-wang}.

As for the converse implication, note that by self-duality of $\fitz$ and $\pi$
 we get
$$ \fitz^*(\xi^*,\xi) = \fitz(\xi,\xi^*)\geq \pi(\xi,\xi^*) = \pi^*(\xi^*,\xi).$$
Thus, $\fitz^* \in {\Fc}(X^*)$ and Proposition \ref{repres} applies.
\end{proof}
\paragraph{\bf Self-dual representatives vs.\ Fitzpatrick  and Penot functions.}
Let $\alpha \subset X \times X^*$ be a cyclically monotone operator
with $\alpha = \partial \psi$  for some  convex and lower
semicontinuous potential  $ \psi:X\rightarrow (-\infty,\infty].$
 As already observed, a self-dual representative of $\alpha$ is the sum of $\psi$ and its convex conjugate.  However in general the Fitzpatrick
   functional
   $f_{\alpha}$ may differ from $\psi+\psi^*$, as shown by the following.
 \begin{example}
\label{multiplicity-rmk}
\upshape Consider $X = \Rr = X^*$ and set $\alpha=$ identity, namely $\alpha = \partial \psi$ with 
$\psi(\xi) = \frac{1}{2}\xi^2$.
The Fitzpatrick function of $\partial \psi$ is $f_{\partial \psi}(\xi,\xi^*) = \xi^2 /4+ \left(\xi^*\right)^2/4 +  \xi\cdot \xi^*/2$, which is not self-dual.
\end{example}
\paragraph{\bf Fitzpatrick and Penot functions in the case of $1$-positively  homogeneous potentials.}
Our next result reveals that, when $\alpha=\partial \psi$ and $\psi: X \to (-\infty,+\infty]$ is
positively homogeneous of degree $1$, then
 also the Fitzpatrick
functional $f_\alpha$
coincides with the bipotential \eqref{example-bipotential}.
\begin{proposition}
\label{bella-prop}
 Let $\psi:X\rightarrow (-\infty,\infty]$ be convex, lower semicontinuous, and  positively
 homogeneous of degree $1$, viz.\ $\psi(\lambda \xi)= \lambda \psi(\xi)$ for all $\xi \in X$
 and $\lambda \geq 0$.

Then, the Fitzpatrick function of the subdifferential of $\psi$ coincides with the sum of $\psi$ and its convex conjugate, i.e.
\begin{equation}
\label{to-be-labeleed}
  f_{\partial\psi}(\xi,\xi^*) = \psi(\xi) + \psi^* (\xi^*)  \qquad \text{for all }(\xi,\xi^*)\in X \times X^*.
\end{equation}
\end{proposition}
\noindent
 Before developing the \emph{proof}, we recall that, for all $\psi$  convex, lower semicontinuous, and $1$-homogeneous there exists
 a closed convex set $0\in K \subset X$ such that $\psi$  coincides with the
 \emph{Minkowski functional} of $K$, viz.\
 \begin{subequations}
 \label{eqns:repre-psi}
 \begin{equation}
 \label{eqns:repre-psi-a}
 \psi(\xi)=
  M_K(\xi)=
  \inf \left\{ \sigma>0\, : \ \frac{\xi}{\sigma} \in K \right \}\,.
  \end{equation}
  Furthermore,
   \begin{equation}
 \label{eqns:repre-psi-b}
 \begin{gathered}
  \psi^{*}(\xi^*) = I_{K^*}(\xi^*) \text{  where $K^*\subset X^*$ is
  the \emph{polar} set of $K$, i.e.}
  \\
  K^*= \{\xi^* \in K^* \, : \ \pi(\xi,\xi^*) \leq 1  \ \forall\, \xi \in K \},
  \end{gathered}
  \end{equation}
  so that
  \begin{equation}\label{useme}
\forall\, (\xi,\xi^*) \in \partial \psi \qquad    M_K(\xi) + I_{K^*}(\xi^*)= \pi(\xi,\xi^*)\,.
\end{equation}
  For later convenience, we also recall that
  \begin{equation}
 \label{eqns:repre-psi-c}
  \psi(\xi) = \sup_{\xi^* \in K^*} \prodpi{\xi^*}{\xi} \quad \text{for all } \xi \in X\,.
  \end{equation}
  \end{subequations}
\begin{proof}[Proof of Proposition \ref{bella-prop}]
It follows from the  definition of the Fitzpatrick  function  \eqref{fitz-formula} and
from  \eqref{useme} that
\[
f_{\partial \psi}(\xi,\xi^*) =
\sup\{ \underbrace{\prodpi{ \xi^*}{ \xi_0 } - M_K(\xi_0)}_{\leq I_{K^*}(\xi^*)}+ \underbrace{\prodpi{  \xi_0^*}{  \xi} - I_{K^*}(\xi_0^*)}_{\leq M_K(\xi)} : \  (\xi_0,\xi_0^*)\in \alpha\}
\]
Hence, we obtain that $f_{\partial \psi}(\xi,\xi^*)\leq M_K(\xi) + I_{K^*}(\xi^*) $ for all
$(\xi,\xi^*) \in X \times X^*$.

 For the opposite inequality assume first   that $\xi^* \notin K^{*}$. Then there exists  $\xi_0 \in K$ such that $\prodpi{\xi^*}{\xi_0} > 1$. Choose an arbitrary $\xi_0^* \in \partial M_K(\xi_0) = \partial M_K( \lambda \xi_0)$
 for any positive $\lambda > 0$. Then, taking into account that $ M_K(\lambda \xi_0)+I_{K^*}(\xi_0^*)= \lambda \pi(\xi_0,\xi_0^*)
 \leq \lambda$,
 we get
\begin{eqnarray*}\lefteqn{\prodpi{ \xi^*}{ \lambda \xi_0 } - M_K(\lambda \xi_0)
 + {\prodpi{  \xi_0^*}{  \xi} - I_{K^*}(\xi_0^*) } } & & \\ & \geq & \lambda \left( \prodpi{ \xi^*}{ \xi_0 }-1\right) + \prodpi{  \xi_0^*}{  \xi} \rightarrow + \infty \qquad \text{ as } \lambda \rightarrow \infty.
\end{eqnarray*}
Therefore $f_{\partial \psi}(\xi,\xi^*) \geq I_{K^*}(\xi^*)$. On the other hand,
taking into account that
 $\partial M_K(0) = K^*$, we  deduce that
\[
 f_{\partial \psi}(\xi,\xi^*) \geq \sup\left\{\prodpi{\xi_0^*}{\xi}: \xi_0^*\in K^* \right\} = M_K(\xi)\,.
\]
Eventually,  we get that $f_{\partial \psi}(\xi,\xi^*)\geq M_K(\xi) + I_{K^*}(\xi^*) $ for all $(\xi,\xi^*) \in X \times X^*$, which concludes the proof.
\end{proof}

\begin{corollary}
\label{cor:singleton}
Let $\psi:X\rightarrow (-\infty,\infty]$ be convex, lower semicontinuous and  positively
 homogeneous of degree $1$.
Then
\begin{equation}
\label{eq:cor2.6}  f_{\partial\psi}(\xi,\xi^*) = \psi(\xi) + \psi^* (\xi^*)= \rho_{\partial\psi^*}(\xi,\xi^*)   \qquad \text{for all }(\xi,\xi^*)\in X \times X^*.
\end{equation}
\end{corollary}
\begin{proof}
To prove \eqref{eq:cor2.6}
we observe that
\[
f_{\partial\psi} = \psi + \psi^* =(\psi + \psi^* )^*=  f_{\partial\psi}^*= \rho_{\partial \psi^*}
\]
where the first identity is due to
Proposition \ref{bella-prop}, the second one to the fact that $\psi + \psi^* $ is self-dual, and the last one to \eqref{2dual-fitz-penot}.
\end{proof}


\subsection{Approximation of maximal monotone operators}
In the next lines,  the symbol
$\mathsf{X}$ may stand for the space $X$, for $X^*$, or for $X \times X^*$.
Let $\mathsf{f}_n$, $\mathsf{f}: \mathsf{X} \to (-\infty,\infty]$ be  convex, proper, and l.s.c.\
functionals
  and let $\alpha_n$, $\alpha \subset X \times X^*$ be maximal monotone operators.
We introduce
the notation
\begin{align}
 \displaystyle\Gammaliminf_{n\to \infty} \mathsf{f}_n (\mathsf{x}) &:= \min\{ \liminf_{n\to \infty}  \mathsf{f}_n(\mathsf{x}_n), \  \mathsf{x}_n \wconv \mathsf{x} \text{ in $\mathsf{X}$}\},\nonumber\\
\displaystyle\Gammalimsup_{n\to \infty} \mathsf{f}_n(\mathsf{x})&:= \min\{\limsup_{n\to \infty}  \mathsf{f}_n(\mathsf{x}_n), \  \mathsf{x}_n \to \mathsf{x}  \text{ in $\mathsf{X}$}\}.\nonumber
\end{align}
These correspond to the classical $\Gammaliminf$ and $\Gammalimsup$ constructions (cf.\ e.g.\
\cite{DalMaso93}),
 with respect to the weak and the strong topology of $\mathsf{X}$, respectively. We will use the following convergence notions, for which the reader
is referred to \cite{Attouch}:
\begin{align}
 \mathsf{f}_n \Mosco \mathsf{f} \quad &\Leftrightarrow \quad \displaystyle\Gammalimsup_{n\to \infty} \mathsf{f}_n\leq \mathsf{f}\leq \displaystyle\Gammaliminf_{n\to \infty} \mathsf{f}_n
 \nonumber\\
\alpha_n \Graph \alpha  \quad &\Leftrightarrow \quad
\forall (\xi,\xi^*) \in \alpha \ \ \exists (\xi_n,\xi_n^*) \in \alpha_n \ : \ \xi_n \to \xi, \ \xi_n^* \to \xi^*.
  \nonumber
\end{align}
where the symbol $\Mosco$ stands for {\it Mosco convergence} in $\mathsf{X}$ and $\Graph$ is usually referred to as {\it graph convergence}. In particular, Mosco convergence corresponds to $\Gamma$-convergence with respect to both the strong and the weak topology of $\mathsf{X}$ and can be made more explicit by
$$
\mathsf{f}_n \Mosco \mathsf{f} \quad \Leftrightarrow \quad
\left\{
  \begin{array}{l}
 \forall \mathsf{x}_n \wconv \mathsf{x}, \quad \mathsf{f}(\mathsf{x}) \leq \displaystyle \liminf_{n \to \infty} \mathsf{f}_n(\mathsf{x}_n)  \\
    \forall \mathsf{x} \in \mathsf{X},\  \exists \mathsf{x}_n  \to \mathsf{x} \ : \ \mathsf{f}_n(\mathsf{x}_n)\to \mathsf{f}(\mathsf{x})
  \end{array}
\right.
$$

In the case of cyclically monotone operators, graph convergence is known to be equivalent to the Mosco convergence of the respective potentials (up to some normalization condition), viz.\ we have the following result.

\begin{theorem}[{\cite[Thm. 3.66, p. 373]{Attouch}}]\label{Attouch}
Let
$\phi_n, \ \phi : X \to (-\infty,+\infty]$ be proper, convex and lower semicontinuous functionals.
The following are equivalent:
  \begin{itemize}
 \item[i)] $\partial \phi_n \Graph \partial \phi$ in
  $X \times X^*$
  and there exist $(\xi_n,\xi_n^*) \in \partial \phi_n$ such that $\xi_n \to \xi $ in $X$, $\xi_n^* \to \xi^*$  in $X^*$, $\phi_n(\xi_n)\to \phi(\xi)$, and $(\xi,\xi^*)\in \partial \phi$;
 \item[ii)] $\phi_n \Mosco \phi$ in $X$.
\end{itemize}
\end{theorem}

The importance of graph convergence is revealed by the following identification lemma, basically consisting of the {\it approximation version} of \cite[Prop. 2.5, p. 27]{Brezis73}.

\begin{lemma}\label{approximation}
  Let $\alpha_n \Graph \alpha$, $(\xi_n,\xi_n^*) \in \alpha_n$, $\xi_n \wconv \xi$, $\xi_n^* \wconv \xi$, and $\liminf_{n \to \infty} \prodpi{ \xi_n^*}{ \xi_n } \leq \prodpi{ \xi^*}{ \xi }$. Then $(\xi,\xi^*) \in \alpha$.
\end{lemma}

In order to prove this, quite classical, approximation lemma, in the particular case of cyclically monotone operators, what is actually needed is just the implication $\, i) \Rightarrow ii) \, $
in Theorem \ref{Attouch}
above. Indeed, let $\partial \phi_n \Graph \partial \phi$, $(\xi_n,\xi_n^*) \in \partial \phi_n$, with $\xi_n \wconv \xi$, $\xi_n^* \wconv \xi^*$, and $\liminf_{n \to \infty}\prodpi{ \xi_n^*}{ \xi_n} \leq \prodpi{ \xi^*}{ \xi }$ be given. We readily have that
$$0 \leq  \phi(\xi) + \phi^*(\xi^*)-\prodpi{ \xi^*}{\xi }  \leq \liminf_{n\to \infty}\big(  \phi_n(\xi_n) + \phi^*_n(\xi_n^*)-  \prodpi{ \xi_n^*}{\xi_n } \big)=0$$
where we have used the fact that $\phi_n^* \Mosco \phi^*$ iff $\phi_n \Mosco \phi$ \cite[Thm.\ 3.18, p.\ 295]{Attouch}, i.e. the bicontinuity of the Legendre-Fenchel transformation with respect to the topology induced by the Mosco convergence. Therefore we conclude that $(\xi,\xi^*) \in \alpha$.

\begin{remark}
\label{rmk:graph_vs_mosco}
As indeed one just needs
$$\partial \phi_n \Graph \partial \phi \quad \Rightarrow \quad \phi \leq \Gammaliminf_{n\to \infty} \phi_n \ \ \text{in} \  \ V \ \ \text{and} \ \  \phi^* \leq \Gammaliminf_{n\to \infty} \phi_n^* \ \ \text{in} \  \ V^* $$
in order to check for Lemma \ref{approximation}, one may wonder if asking directly the two $\Gammaliminf$ conditions above would weaken the convergence requirements on the functionals. This is however not the case. Indeed, under some very general condition of {\it equi-properness} type, we have that the two separate $\Gammaliminf$ conditions  are indeed equivalent to $\phi_n \Mosco \phi$ \cite[Lemma 4.1]{be} and hence entail $\partial \phi_n \Graph \partial \phi$.
\end{remark}

Our next aim is that of extending the above arguments to the case of noncyclically maximal monotone operators. In particular, we present an extension of
 Theorem \ref{Attouch} in terms
  of representative functions,
   and in particular of the Fitzpatrick and of the Penot functionals.

 \begin{theorem}\label{Attouch2}
 Let $(\alpha_n ),\, \alpha$ be  maximal monotone operators $\alpha_n, \, \alpha : X \rightrightarrows X^*$.
 The following are equivalent:
   \begin{itemize}
 \item[i)] $\alpha_n \Graph \alpha$  \ \ \text{in $X \times X^*$},
 \item[ii)] $f_\alpha \leq \displaystyle\Gammaliminf_{n\to \infty} f_{\alpha_n} $ \ \ \text{in $X \times X^*$},
  \item[iii)]
$\Gammalimsup_{n\to \infty} \rho_{\alpha_n} \leq \rho_\alpha$ \ \ \text{in $X \times X^*$}.
\end{itemize}
  \end{theorem}

  Exactly as  above, the proof of Lemma \ref{approximation} follows just from  implication $i) \Rightarrow ii)$. We shall however give a full equivalence proof for the sake of completeness and comparison with Theorem \ref{Attouch}. In particular, note that Condition $ii)$ above is weaker than $f_{\alpha_n}\Mosco f_\alpha$.  
That is to say that the former Theorem \ref{Attouch} does not follow directly from Theorem \ref{Attouch2}.

  \begin{proof}
  \textbf{Claim $1$: $i) \Rightarrow ii)$.}  Fix $(\xi_0,\xi_0^*)\in \alpha$ and let $(\xi_{0n}, \xi_{0n}^*)\in \alpha_n$ be such that $\xi_{0n} \to \xi_0$ in $X$ and $\xi_{0n}^* \to \xi_0^*$ in $X^*$. Moreover, let $\xi_n \wconv \xi$
   in $X$
   and $\xi_n^* \wconv \xi^*$ in $X^*$. We have that
   \begin{align}
   \liminf_{n\to \infty}  f_{\alpha_n}(\xi_n,\xi_n^*)&\geq \liminf_{n\to \infty}  \left( \prodpi{ \xi_n^*}{ \xi_{0n}} - \prodpi{ \xi_{0n}^*}{\xi_{0n} - \xi_n } \right) \nonumber\\
&= \prodpi{ \xi^*}{ \xi_{0}} - \prodpi{ \xi_{0}^*}{\xi_{0} - \xi }.\nonumber
   \end{align}
In particular, by passing to the supremum with respect to  $(\xi_0,\xi_0^*)\in \alpha$, we conclude that $f_\alpha \leq \Gammaliminf_{n\to \infty}f_{\alpha_n}$.

\noindent
\textbf{Claim $2$: $i) \Rightarrow iii)$.}
Observe that  $\alpha_n \Graph \alpha$  implies $\alpha_n^{-1} \Graph \alpha^{-1} $, hence Claim $1$
yields $f_{\alpha^{-1}} \leq \displaystyle\Gammaliminf_{n\to \infty} f_{\alpha_n^{-1}} $.
By convex conjugation and
taking into account \eqref{2dual-fitz-penot} and \cite[Thm.\ III.3.7, p.\ 271]{Attouch}, we then have
\[
\rho_\alpha= (f_{\alpha^{-1}})^* \geq \left( \Gammaliminf_{n\to \infty} f_{\alpha_n^{-1}}  \right)^*
= \Gammalimsup_{n\to \infty} \rho_{\alpha_n}\,.
\]

\noindent
\textbf{Claim $3$: $iii) \Rightarrow i)$.}
 Fix $(\xi,\xi^*)\in \alpha$ and let $(\xi_n,\xi_n^*)\in \alpha_n$
 fulfill $\xi_n \to \xi$ in $X$, $\xi_n^* \to \xi^*$ in $X^*$, and
$$\limsup_{n \to \infty}\rho_{\alpha_n}(\xi_n,\xi_n^*) \leq \rho_\alpha(\xi,\xi^*)=\prodpi{ \xi^*}{ \xi}$$
(such sequences exist as $\Gammalimsup_{n\to \infty}\rho_{\alpha_n} \leq \rho_\alpha$). In particular, we have that
$$\rho_{\alpha_n}(\xi_n, \xi_n^*) <\prodpi{ \xi^*}{ \xi} + \veps_n $$
for some sequence $\veps_n \to 0$. By exploiting the extension of the {\it Br\o nsted-Rockafellar} approximation
Lemma from \cite[Thm. 3.4]{Marques-Alves-Sveiter08}, we have that there exist $(\tilde \xi_n,\tilde \xi_n^*)\in \alpha_n$ such that
 for all $n \in \N$
$$ \| \xi_n - \tilde \xi_n \|^2 \leq {\veps_n}, \quad \| \xi_n^* - \tilde \xi_n^* \|_*^2 \leq {\veps_n}.$$
Then, a  classical  diagonal-extraction  argument yields $ \tilde \xi_n \to \xi$
in $X$ and
$\tilde \xi_n^* \to \xi^*$ in $X^*$.

\noindent
\textbf{Claim $4$: $ii) \Rightarrow i)$.}
Again by convex conjugation, and \eqref{2dual-fitz-penot}, we deduce from \emph{ii)} that
\[
\rho_{\alpha^{-1}} \geq
 \Gammalimsup_{n\to \infty} \rho_{\alpha_n^{-1}}.
\]
Therefore,  in view of Claim $3$ we have that $\alpha_n^{-1} \Graph \alpha^{-1} $, whence $\alpha_n \Graph \alpha $.
  \end{proof}

\section{\bf Setup and preliminary results} \label{s:mainresults}
\noindent
Before stating our working assumptions in Sec.\ \ref{s:4.2}, 
in the upcoming Sec.\ \ref{s:4.1} we recall all the basic definitions, and  tools of
measure theory and  convex analysis, which we will use in the following.
\subsection{Preliminaries of measure theory,  $\BV$ functions, and convex analysis}
\label{s:4.1}
\noindent We start with the notion of measure with values in a Banach space $\mathsf{X}$, which later on will either coincide with
the reflexive space $X$, or with $\R$.
\begin{definition}[Vector measure]
 Let $(\Omega,\Sigma)$ be a measurable space. 
  A function $\mu:\Sigma \rightarrow \mathsf{X}$ is called a (Banach-space valued) vector measure, if
 \begin{equation}
  \forall \left(A_i\right)_{i\in\Nn}, A_i\in \Sigma
\text{ with }
  \left( i\neq j \Rightarrow A_i \cap A_j = \emptyset \right) \quad \text{it holds} \quad \mu \left(\bigcup_{i\in\Nn} A_i\right) = \sum_{i=1}^{\infty} \mu\left(A_i\right)\,.  \label{eq:bvalmeas-sigma-add}
 \end{equation}
Here the convergence of the series on the right-hand side has to be understood in terms of the norm of  $\mathsf{X}$.
\end{definition}

\begin{definition}[Variation of a measure] Let $(\Omega, \Sigma)$ be a measurable space 
 and $\mu:\Sigma\rightarrow \mathsf{X}$ a vector measure. Then the variation of $\mu$, denoted by $\|\mu\|:\Sigma \rightarrow [0,\infty]$, is given by
\begin{equation}
 \|\mu\|(A) := \sup\left\{\sum_{i=1}^\infty \left\|\mu(A_i)\right\|\, : (A_i)_{i\in\Nn} \subset \Sigma\,,\ \bigcup_{i\in\Nn}A_i = A\,, \ \forall i\neq j: A_i\cap A_j = \emptyset \right\}
\end{equation}
for all $A\in \Sigma$. If $\|\mu\|(\Omega) < \infty$ then we say $\mu$ is of bounded variation.
\end{definition}
\noindent Indeed, $\|\mu\|$ itself is a (positive) measure on $(\Omega,\Sigma)$, see \cite[Prop.\ 9, p.\ 3]{DU77}.

\begin{definition}[Absolute continuity and singularity of measures]
\label{def:ac-singularity-meas}
 Let $(\Omega,\Sigma)$ be a measurable space,
  $\mu:\Sigma \rightarrow \mathsf{X}$  a vector measure, and $\nu:\Sigma \rightarrow [0,\infty]$ be a (real-valued, positive)
   measure. We say that $\mu$ is absolutely continuous w.r.t.\ to $\nu$, and write $\mu \ll \nu$, if
\begin{equation}
 \forall A \in \Sigma: \Big[ \ \nu(A) = 0 \quad \Longrightarrow \quad \mu(A) = 0 \ \Big]\,.
\end{equation}
Moreover, we say that two real-valued, positive measures $\mu$ and $\nu$ are singular,
and write $\mu \perp \nu$,
 if there exist $B_1, B_2 \in \Sigma$ with
$B_1\cup B_2 = \Omega$ and $B_1\cap B_2 = \emptyset$ such that
\begin{equation}
 \forall A\in \Sigma: \quad \quad \mu(A) =\mu\left(A\cap B_1\right) \quad \text{and} \quad \nu(A) = \nu\left(A\cap B_2\right)
\end{equation}
\end{definition}
\noindent We recall the following  generalization of the Lebesgue decomposition theorem, see e.g.\
 \cite[Thm.\ 9, p.\ 31]{DU77}. 
\begin{theorem}[Lebesgue decomposition theorem]\label{thm:Leb-decomp}
 Let $(\Omega, \Sigma)$ be a measure space, $\sigma$ be a Banach space-valued measure of bounded variation and $\lambda$ a real valued, positive measure.
 Then there exist two unique vector measures $\sigma_{\mathrm{ac}}$, $\sigma_\sing$ on $(\Omega,\Sigma)$, which are of bounded variation, such that
\begin{equation}
\|\sigma_{\mathrm{ac}}\| \ll \lambda\,, \ \|\sigma_\sing\| \perp \lambda
\quad \text{and} \quad  \sigma = \sigma_{\mathrm{ac}} +\sigma_\sing \,.
\end{equation}
\end{theorem}
\paragraph{\bf   $\BV$  functions.} We fix here some definitions and notation concerning
$\BV$-functions  on $[0,T]$ with values in a Banach space $\mathsf{X}$, referring e.g.\  to \cite{moreau88} for a comprehensive
introduction to the topic.
We denote by $\BV ([0,T];\mathsf{X})$ the space of the measurable, pointwise defined at every
time $t \in [0,T]$,
 functions $v :[0,T] \to \mathsf{X}$ such that their \emph{pointwise} total variation on $[0,T]$ is finite, i.e.\
 \[
 \mathrm{Var} (v; [0,T]) = \sup \left\{ \sum_{m=1}^M \| v(t_m) - v(t_{m-1})\| \, : \ 0 =t_0 < t_1 < \ldots < t_{M-1}<t_M=T \right\}<\infty\,.
 \]
 More in general, given a convex, lower semicontinuous, $1$-positively homogeneous 
 functional $\psi: X \to [0,+\infty)$, we
 denote by
$ \mathrm{Var}_{\psi}$ the induced total variation, i.e.
\begin{equation}
\label{var-psi}
 \mathrm{Var}_\psi (v; [0,T]) = \sup \left\{ \sum_{m=1}^M \psi( v(t_m) - v(t_{m-1})) \, : \ 0 =t_0 < t_1 < \ldots < t_{M-1}<t_M=T \right\}<\infty\,.
 \end{equation}

It is well known that the distributional derivative $\dd v$ of a curve $v \in \BV (0,T;X)$ is a vector measure in $\Mm(0,T;\mathsf{X})$, where
\[
\Mm(0,T;\mathsf{X}) = \{ \text{ Radon vector measures $\mu : (0,T) \to \mathsf{X}$ with bounded variation } \},
\]
which we will endow with the weak$^*$-topology.
\begin{notation}
\upshape
Let $u \in \BV ([0,T];X)$.
 Applying Thm.\ \ref{thm:Leb-decomp} with the choices
 $\sigma = \dd u$ and $\lambda=\Ll $ (where $\Ll$ denotes the one-dimensional Lebesgue measure on $[0,T]$), we find that there
  exist  vector measures
  $(\dd u)_{\mathrm{ac}},\, (\dd u)_\sing \in \Mm(0,T;X)$  such that
 \begin{equation}
  \|(\dd u)_{\mathrm{ac}}\| \ll \Ll,  \quad   \|(\dd u)_\sing\| \perp \Ll \ \ \text{and} \ \ \dd u
  = (\dd u)_{\mathrm{ac}} + (\dd u)_\sing\,. \label{decomp-lam}
 \end{equation}
 Thanks to the Radon-Nikod\'ym property of the \emph{reflexive} space $X$, the Radon-Nikod\'ym derivatives
 \begin{equation}
 \label{decomp-u}
  \dot{u}_{\mathrm{ac}}(t) := \frac{(\dd u)_{\mathrm{ac}}}{\dd \Ll}\,,\ \dot{u}_\sing(t) := \frac{(\dd u)_\sing}{ \|(\dd u)_\sing\|} \qquad \text{exist }\foraa\, t \in (0,T).
 \end{equation}

 For later use, we remark that
for any convex, lower semicontinuous, $1$-positively homogeneous $\psi: X \to [0,+\infty)$ there holds

\begin{equation}\label{wil-be-used-later}
\mathrm{Var}_\psi (u; [0,T])= \int_0^T \psi(\dot{u}_{\mathrm{ac}}(t)) \dd t + \int_0^T \psi (\dot{u}_\sing(t)) \|(\dd u)_\sing\| (t)\,.
\end{equation}
 \end{notation}

\paragraph{\bf The recession function.}
Finally, we recall the concept of \emph{recession function}  (see \cite[Chap.\ 4]{FL07}).  Note that the following definitions and
results, which  are stated in \cite{FL07} for convex functions on $\R^m$, in fact extend to an infinite-dimensional setting,
as it can be easily checked.
    \begin{definition}[Recession function]
\label{def-recession function}
 Let $\mathsf{X}$ be a vector space and $g: \mathsf{X} \rightarrow (-\infty, \infty]$ be a convex functional. Its recession function $g^{\infty}$ is defined as
\begin{equation}
\label{def:recession-function}
 g^\infty(z) := \sup\left\{g(y+z)-g(y): y\in D(g)\right\}\,.
\end{equation}
\end{definition}
\noindent
Trivially adapting the argument from  \cite[Thm.\ 4.70, p.\  290]{FL07}, it can be shown that
that $g^\infty$ is  positively homogeneous of degree $1$ and  convex. Moreover,
if $g$ is lower semicontinuous, so is $g^\infty$. Furthermore,  there holds
\begin{equation}
\label{representation-recession}
g^\infty (z) = \lim_{t \to \infty} \frac{g(y+tz)-g(y)}{t}= \sup_{t >0} \frac{g(y+tz)-g(y)}{t} \quad \text{for every } y \in D(g)\,.
\end{equation}

In what follows, we will denote by $f_\alpha^\infty$
the recession function of the Fitzpatrick function $f_\alpha$, viz.\
\begin{equation}
\label{recession-fitz}
f_\alpha^\infty(\xi,\xi^*)= \sup\left\{f_\alpha(\xi + x,\xi^*+x^*)-f_\alpha(x,x^*): (x,x^*)\in D(f_\alpha)\right\}\,.
\end{equation}
We now prove a useful representation formula for
$f_\alpha^\infty$, cf.\  \cite[Prop.\ 4.77, p.\ 294]{FL07}.
\begin{lemma}
\label{l:repre-recession}
There holds
\begin{equation}
\label{repre-recession}
f_\alpha^\infty(\xi,\xi^*)  = \sup\{   \scalx{\xi^*,  \xi_0 } + \scalx{  \xi_0^*, \xi}\, : \  (\xi_0,\xi_0^*)\in D(\rho_{\alpha^{-1}})\}  \quad \forall\,(\xi,\xi^*) \in X \times X^*\,.
\end{equation}
\end{lemma}
\begin{proof}
Following
the proof of \cite[Prop.\ 4.77]{FL07},
from \eqref{representation-recession} and \eqref{1dual-fitz-penot} we infer
\[
\begin{aligned}
f_\alpha^\infty(\xi,\xi^*)  & = \sup_{t  >0} \frac{f_\alpha( x+t\xi, x^* +t\xi^*)  -f_\alpha(x,x^*) }{t}
\\ & \geq
\sup_{t > 0} \frac1t \Big(  t \scalx{\xi^*,\xi_0 } + t \scalx{\xi_0^*,\xi }
\\ & \quad \quad  \qquad
+ \scalx{\xi_0^*,x } + \scalx{x^*,\xi_0 } - \rho_{\alpha^{-1}}(\xi_0,\xi_0^*) - f_\alpha(x,x^*)    \Big)
\\
&
\geq \scalx{\xi^*,\xi_0 } +  \scalx{\xi_0^*,\xi }
\\ & \quad
 +\frac1t \left( \scalx{\xi_0^*,x } + \scalx{x^*,\xi_0 } - \rho_{\alpha^{-1}}(\xi_0,\xi_0^*) - f_\alpha(x,x^*)   \right)
\ \forall\,(\xi_0,\xi_0^*) \in X \times X^*, \ t >0\,.
\end{aligned}
\]
In view of \eqref{1dual-fitz-penot}, we thus conclude that
\[
f_\alpha^\infty(\xi,\xi^*)  \geq  \sup\{   \scalx{\xi^*,  \xi_0 } + \scalx{  \xi_0^*, \xi}\, : \  (\xi_0,\xi_0^*)\in D(\rho_{\alpha^{-1}})\} \,.
\]
The converse inequality may be proved arguing along the very same lines, cf.\ also  the proof of  \cite[Prop.\ 4.77]{FL07}.
\end{proof}
\noindent
As a direct consequence of Lemma \ref{l:repre-recession}, we have
  the following representation formula
 for the recession function of  $f_\alpha$, in the
case $\alpha$ is the subdifferential of a $1$-positively homogeneous
potential.
\begin{corollary}
Let $\psi:X\rightarrow \R$ be convex, lower semicontinuous and  positively
 homogeneous of degree $1$, and let $K^* \subset X^*$ be the associated polar set, cf.\ \eqref{eqns:repre-psi}. Then,
 \begin{equation}
 \label{recession-1-homog}
 f_{\partial \psi}^\infty (\xi,\xi^*)  =  \sup \{ \scalx{\xi_0^*,\xi} +  \scalx{\xi^*,\xi_0} \, : \ (\xi_0,\xi_0^*) \in X \times  K^*  \}\,.
 \end{equation}
\end{corollary}
\begin{proof}
Formula \eqref{recession-1-homog} follows from \eqref{repre-recession}, taking into account that
\[
\rho_{\alpha^{-1}}= \rho_{\partial \psi^*}  = \psi + \psi^* =
\psi+ I_{K^*},
\]
and that $D(\psi)=X$ by assumption.
\end{proof}

\subsection{\bf Basic assumptions}
\label{s:4.2}
\noindent
In what follows,
we will suppose that
 \begin{equation}
 \label{X-reflexive}
\text{$X$ is a reflexive Banach space}
 \end{equation}
 and that
\begin{equation}
 \label{gen-max-monot}
 \tag{\ref{s:mainresults}.$\alpha_0$}
\alpha:X\rightrightarrows \Xs \text{ is a maximal monotone operator with  }   0 \in \alpha(0).
 \end{equation}
As for the energy functional $\cE$,  along the lines of  \cite{MRS11x}
 we require the following  \emph{coercivity} and \emph{regularity} type conditions.  Recall that $\frsubname$ denotes the \emph{Fr\'echet subdifferential} of the map
 $u \mapsto \Ec_t(u)$, cf.\ \eqref{frsub-def}.
\begin{assum}[Assumptions on the energy] \label{ener}
 We assume that the pair  $(\Ec,\frsubname)$ has the following properties:\\
\begin{itemize}
\item[\textbf{Lower semicontinuity:}] The domain of $\Ec$ is of the form  $D(\Ec) = [0,T]\times D$
for some $D \subset X$,
 and $\frsubname:[0,T]\times D \rightrightarrows X^*$. Furthermore, we ask that
\begin{equation}
 \label{cond:E0} \tag{$\ref{s:mainresults}.\Ec_0$} \begin{array}{c}
u\mapsto \Ec_t(u) \ \text{is l.s.c. for all}\ t\in[0,T],\ \ \exists C_0>0: \ \forall (t,u)\in[0,T]\times D: \Ec_t(u)\geq C_0
\ \text{and}\\ \graph(\frsubname)\ \text{is a Borel set of}\ [0,T]\times X\times X^*\,.\end{array}
\end{equation}
 \item[\textbf{Coercivity:}]
 Set $\cg(u):= \sup_{t \in [0,T]} \ene tu$ for every $u \in D$.
 We require that
\begin{equation}
  \label{cond:E1} \tag{$\ref{s:mainresults}.\Ec_1$}
u\mapsto \cg(u) \ \text{\normalfont{has compact sublevels.}}
\end{equation}
 \item[\textbf{Time-differentiability:}] For any $u\in D$ the map $t\mapsto \Ec_t(u)$ is differentiable with derivative $\partial_t \Ec_t(u)$ and it holds
\begin{equation}
 \label{cond:E2} \tag{$\ref{s:mainresults}.\Ec_2$} \exists C_1 >0: \forall u \in D: \ \ \left| \partial_t \Ec_t(u)\right| \leq C_1 \Ec_t (u)
\end{equation}
 \item[\textbf{Weak closedness:}] For all $t\in[0,T]$ and for all sequences $(u_n)_{n\in\Nn} \subset X$, $\xi_n \in \diff t{u_n}$, $E_n = \Ec_t(u_n)$ and $p_n=\partial_t \ene{t}{u_n}$ with
\[
 u_n \rightarrow u\ \text{in}\ X, \quad \xi_n \wconv \xi\ \text{in}\ X^*, \quad   
 p_n \rightarrow p, \ \text{ and } \ {E}_n \rightarrow {E}\ \text{in}\ \Rr
\]
 it holds
\begin{equation}
 \label{cond:E5} \tag{$\ref{s:mainresults}.\Ec_3$} (t,u) \in D(\frsubname), \ \xi \in \diff tu,\ p\leq \partial_t \ene tu\ \text{ and } \ E=\Ec_t(u)\,.
\end{equation}
\end{itemize}
\end{assum}
\begin{remark}
\upshape
\label{rmk:observation}
In fact,  
 up to a translation, we may always suppose that the
constant involved in \eqref{cond:E0}  is strictly positive.
As in \cite{MRS11x}, combining  \eqref{cond:E2} with the Gronwall Lemma
we observe  that
\begin{equation}
\label{consequence-of-gronwall}
\exists\, C >0 \qquad \forall\, (t,u) \in [0,T]\times D \qquad  \cg(u) \leq C \inf_{t \in [0,T]} \Ec_t(u).
\end{equation}
\end{remark}
\medskip

\noindent
Later on, Assumption \ref{ener}
 will be  complemented by a suitable version of the chain rule for $\cE$, cf.\ Assumption \ref{ass:chain-rule}
 below.
 As already mentioned, in order to investigate the stability properties of the doubly nonlinear
 equation
 \begin{equation}
 \label{eq:dne} \alpha(\dot u(t)) + \diff {t}{u(t)} \ni 0 \text{ in } X^*\quad \text{for a.a.}\ t\in(0,T)\,,
\end{equation}
  under graph convergence of $\alpha$,
 it is essential to resort to the Fitzpatrick function $f_\alpha$
 associated with $\alpha$.
 In the following lines,
 we will
 therefore shed light on  how
 \eqref{eq:dne} can be  in fact reformulated in terms of an \emph{energy identity} (cf.\ \eqref{eq:defsol}
 below) featuring $f_\alpha$.
At first, we will  confine the discussion to the
  case of \emph{absolutely continuous} solutions $u$ to
 \eqref{eq:dne}.
 \paragraph{\bf Reformulations of \eqref{eq:dne} in the \emph{absolutely continuous} case.}
 Preliminarily, let us precisely define
  what we understand by an \emph{absolutely continuous} solution to \eqref{eq:dne}.
  \begin{definition}[Absolutely continuous solution]
  \label{def:solutions-AC_case}
  In the framework of \eqref{X-reflexive}, \eqref{gen-max-monot}, and \eqref{cond:E0}, we say that
 a
 curve $u \in W^{1,1}(0,T;X) $ is a  solution to \eqref{eq:dne}, if there exists
  $\xi \in  L^1(0,T;X^*)$  with
 \begin{equation}
 \label{def-sol-touple}
 \xi(t) \in ({-}\alpha(\dot u (t))) \cap \diff t {u(t)} \qquad \foraa\, t \in (0,T)\,.
 \end{equation}
  \end{definition}
 \noindent In what follows, with a slight abuse of notation we will
  sometimes say that $(u,\xi)$ is a solution to \eqref{eq:dne}, meaning that \eqref{def-sol-touple} holds.

In Proposition \ref{prop:sol}, we reformulate
\eqref{def-sol-touple} by means of an energy identity involving the Fitzpatrick function
$f_\alpha$.
In the proof, a key role is played by the chain-rule condition \eqref{ch-rule-AC} below on the energy $\cE$,
whereas note that not all of the conditions collected in Assumption \ref{ener} are needed.
\begin{proposition}[Variational reformulation]
\label{prop:sol}
In the framework of \eqref{X-reflexive},
let $\alpha: X \rightrightarrows \Xs$ fulfill \eqref{gen-max-monot}
and suppose that $\cE: [0,T] \times X \to (-\infty,+\infty]$ complies with
\eqref{cond:E0},
\eqref{cond:E1},
\eqref{cond:E2},
 and the following chain rule: for every
$u\in W^{1,1}(0,T;X)$ and $\xi\in L^1(0,T;X^*)$ such that
\[
\begin{gathered}
    \sup_{t\in[0,T]} \Ec_t (u(t)) < \infty,  \ \ \xi(t)\in \diff t{u(t)} \ \text{for a.a.} \ t\in(0,T)\,,
    \ \
 \int_0^T f_\alpha(\dot u(t),-\xi(t)) \dd t <\infty, 
\end{gathered}
\]
(observe that, thanks to \eqref{cond:E2},  the first of the conditions above guarantees
$\int_0^T |\partial_t \ene t{u(t)}| \dd t <\infty$ as well),
there holds
\begin{equation}
\label{ch-rule-AC}
\begin{gathered}
\text{the map } t \mapsto \ene t{u(t)} \text{ is absolutely continuous and}
\\
\frac{\dd}{\dd t} \ene t{u(t)} =\scalx{\xi(t),\dot u(t)} +\partial_t \ene t{u(t)} \qquad \foraa\, t \in (0,T).
\end{gathered}
\end{equation}
Then, the following implications hold:
\begin{enumerate}
\item if  $(u,\xi) \in W^{1,1}(0,T;X) \times L^1(0,T;X^*)$ fulfills
 the energy identity
\begin{equation}
 \Ec_t(u(t)) + \int_{0}^t{f_{\alpha}\left(\dot{u}(s),-\xi(s)\right) \dd s} = \Ec_0(u(0)) + \int_0^t{\partial_t \Ec_s(u(s))\dd s} \qquad \text{for all } t \in (0,T],
  \label{eq:defsol}
\end{equation}
then $(u,\xi) $ is a solution to \eqref{eq:dne} in the sense of
Def.\ \ref{def:solutions-AC_case}.
\item every solution $(u,\xi)$ to \eqref{eq:dne} (in the sense of Def.\ \ref{def:solutions-AC_case}) fulfilling
\begin{equation}
\label{e:technical-condition}
\sup_{t\in[0,T]} \Ec_t (u(t)) < \infty, \qquad
\int_0^T |\scalx{\xi(t),\dot u (t)}| \dd t <\infty,
\end{equation}
complies  in addition with the energy identity \eqref{eq:defsol}.
\end{enumerate}
\end{proposition}
\noindent
Observe that, for every solution  $(u,\xi)$ to \eqref{eq:dne},  since $-\xi \in \alpha (\dot u)$ a.e.\
in $(0,T)$ and $0 \in \alpha (0)$,  we have $\langle -\xi,\dot u\rangle \geq 0$ a.e.\
in $(0,T)$. Hence
the second of \eqref{e:technical-condition}
in fact reduces to  $ \int_0^T \langle -\xi,\dot u\rangle \dd t <\infty$.
\begin{proof}
 Let $(u,\xi)$ fulfill
 \eqref{eq:defsol}.
 Taking into account that $f_\alpha(\dot u,-\xi) \geq 0$ a.e.\ in $(0,T)$ thanks to \eqref{to-be-observed-later},
 and exploiting \eqref{cond:E2} we gather
 \begin{equation}
 \Ec_t(u(t)) \leq  \Ec_0(u(0)) +  C_1 \int_0^t \Ec_s(u(s))\dd s \qquad \text{for all } t \in (0,T],
  \label{gronwall-estimate}
\end{equation}
whence $\sup_{t \in [0,T]} \ene t {u(t)}<\infty$. Therefore, a fortiori \eqref{eq:defsol} yields
$f_\alpha (\dot u,-\xi) \in L^1 (0,T)$. Hence the pair $(u,\xi)$ fulfills the
conditions
for the chain rule \eqref{ch-rule-AC}, which
yields for all $t \in (0,T]$
\begin{equation}
\label{recalled-later}
\begin{aligned}
 \int_0^t{f_{\alpha}\left(\dot{u}(s),-\xi(s)\right) \dd s} & \leq  \Ec_0(u(0)) - \Ec_t(u(t)) + \int_0^t {\partial_t \ene s{u(s)}\dd s} \\
& \leq  \int_0^T{\scalx{-\xi(s),\dot u(t)}\dd s}\,.
\end{aligned}
\end{equation}
Using that $f_\alpha$ represents $\alpha$, it is immediate to deduce from the above inequality that
 $-\xi(t) \in \alpha\left(\dot u(t)\right)$ for almost all $t\in(0,T)$, thus
 $(u,\xi)$ is  a solution to \eqref{eq:dne} in the sense of Def.\ \ref{def:solutions-AC_case}.

 Conversely, let  $(u,\xi) \in  W^{1,1}(0,T;X) \times L^1(0,T;X^*)$ be a solution
 to \eqref{eq:dne} (in the sense of Def.\ \ref{def:solutions-AC_case})
 fulfilling in addition \eqref{e:technical-condition}. Then, since
 $f_\alpha (\dot u,-\xi) = \scalx{-\xi,\dot u}$, the chain rule
 \eqref{ch-rule-AC} applies, yielding, for all
$t\in[0,T]$, the energy identity
\[
\begin{aligned}
 \int_0^t{f_{\alpha}\left(\dot{u}(s),-\xi(s)\right) \dd s} & = \int_0^t{\scalx{-\xi(s),\dot u(t)}\dd s} \\
 & =\Ec_0(u(0)) - \Ec_t(u(t)) + \int_0^T {\partial_t \Ec_s (u(s)),\xi(s))\dd s} \,.
\end{aligned}
\]
\end{proof}
\begin{remark}
\upshape
\label{rmk:towards-ineqs}
A few comments on Proposition \ref{prop:sol} are in order.
\begin{enumerate}
\item It is not difficult to check that in Proposition \ref{prop:sol}
the Fitzpatrick function $f_\alpha$ could be replaced by \emph{any} representative functional for $\alpha$.
\item
Observe that, in the chain of inequalities \eqref{recalled-later} leading to
the proof of part (1) of Proposition \ref{prop:sol},   it is  in principle necessary for
\eqref{eq:defsol} and for the chain rule \eqref{ch-rule-AC} to hold as inequalities, only.
The proof of part (2)  requires \eqref{ch-rule-AC} to hold as an
equality, instead.
\end{enumerate}
\end{remark}
\section{\bf Main results}
\label{s:4.3-new}
\noindent  Before stating Thm.\ \ref{thm:main-ref}, let us precise  our hypothesis on
the sequence $(\alpha_n)$ of maximal monotone operators.
\begin{assum}\label{assum:alpha-ref}
    Let  $\alpha_n: X\rightrightarrows \Xs$
 fulfill \eqref{gen-max-monot}
 for all $n\in \Nn$ and
    \begin{equation}
    \tag{\ref{s:mainresults}.$\alpha_1$}
    \begin{aligned}
  \exists\, c_1,c_2,c_3>0,  \quad p\geq 1, \quad  q>1
  \qquad  & \forall\, n \in \N
\quad  \forall  (x,y)\in \alpha_n\,: \\ &  \scalx{y,x} \geq c_1\|x\|^p+ c_2\|y\|_*^q - c_3\,.
 \end{aligned}
 \label{eq:alpha_0}
    \end{equation}
Furthermore, there exists $\alpha: X\rightrightarrows \Xs$ fulfilling
\eqref{gen-max-monot} such that
$\alpha_n\Graph \alpha$.
\end{assum}
\begin{remark}
\upshape
\label{straight-conse}
Combining \eqref{eq:alpha_0} with the graph convergence of $(\alpha_n)$ to $\alpha$, it is immediate to conclude 
\begin{equation}
\label{coerc-alpha}
  \scalx{y,x} \geq c_1\|x\|^p+ c_2\|y\|_*^q - c_3 \qquad \text{for all }  (x,y)\in \alpha\,.
\end{equation}
\end{remark}
\noindent
The following example guarantees that
our analysis encompasses the $p_n \to 1$ vanishing-viscosity limit.
\begin{example}
\label{ex:van-visc-lim}
\upshape
Let   $(p_n) \subset [1,+\infty)$
fulfill
$p_n \downarrow 1$ as $n \to \infty$, and let us set
\[
\psi_n(x) = \frac1{p_n} \| x\|^{p_n}, \qquad \alpha_n= \partial \psi_n: X \rightrightarrows X^*\,.
\]
Clearly, $(\psi_n)$ Mosco-converges to $\psi(x) = \| x \|$, hence
$(\alpha_n)$ converges in the  sense of graphs to $\alpha=\partial
\psi$. Observe that $\psi_n^*(y) = \frac1{q_n} \| y\|_*^{q_n}$ with
$q_n= p_n/ (p_n-1) \in [2,\infty]$ for all $n \in \N$, and that
\[
\scalx{y,x}= \frac1{p_n} \| x\|^{p_n} +  \frac1{q_n} \| y\|_*^{q_n}= \| x\|^{p_n} =  \| y\|_*^{q_n}
\quad \text{for all } (x,y) \in \alpha_n\,.
\]
Therefore, Assumption \ref{assum:alpha-ref}
 is  satisfied.
\end{example}


The main result of this section addresses the passage to the limit
as $n \to \infty$ in the doubly nonlinear equations
\begin{equation}
 \label{eq:dne-n} \alpha_n(\dot u(t)) + \diff t{u(t)} \ni 0 \text{ in } X^*\quad \text{for a.a.}\ t\in(0,T)\,.
\end{equation}
In particular, we will assume to be given a sequence $(u_n)$ of
absolutely continuous solutions to \eqref{eq:dne-n}  and we will show
that, if the sequence $(\alpha_n)$
complies with Assumption \ref{assum:alpha-ref},
up to a subsequence $(u_n)$  converges to a curve $u$ fulfilling a
suitable generalized formulation of \eqref{eq:dne}.

 Observe that,
\eqref{eq:alpha_0} in principle only allows for a bound of the type
$\|\dot{u}_n\|_{L^1 (0,T;X)} \leq C$. That it why, we can only
expect a $\BV([0,T];X)$-regularity for the limiting curve $u$, and
\eqref{eq:dne} has to be  weakly formulated accordingly. This
will be done through an energy \emph{inequality} akin to
\eqref{eq:defsol}, cf.\ \eqref{eq-fitzpatrick-BV} below. Therein,
 suitable replacements of the
``time-derivative'' of $u$ are suitably handled in terms of the
Fitzpatrick function
  $f_\alpha$ and of its recession
function $f_\alpha^\infty$ (cf.\ Definition \ref{def-recession
function}), and of  the absolutely continuous and singular
parts of the Radon derivative $\dd u$ of $u$. Having in
mind the role of the chain rule \eqref{ch-rule-AC} relating
\eqref{eq:dne} and the energy identity \eqref{eq:defsol}, it is to
be expected that a suitable \emph{$\BV$ version} of
\eqref{ch-rule-AC} will play a relevant role.
 We state it in the following:
\begin{assum}
\label{ass:chain-rule}
 Let $u\in \BV([0,T];X)$ and $\xi\in L^1(0,T;X^*)$
fulfill
\[
\begin{gathered}
    \sup_{t\in[0,T]} \Ec_t (u(t)) < \infty,  \ \ \xi(t)\in \diff t{u(t)} \ \text{for a.a.} \ t\in(0,T)\,,
    \ \
 \int_0^T f_\alpha(\dot u(t),-\xi(t)) \dd t <\infty, 
\end{gathered}
\]
 and suppose that
the map $ t\mapsto \Ec_t(u(t))$ is almost everywhere equal on
 $(0,T)$ to a function
 $E \in \BV ([0,T])$.
 Furthermore let $\dd u$ and  $\dd E$   denote the Radon derivatives of $u$ and $E$.

 Then,
 for almost all Lebesgue points $t_0$ of the absolutely continuous parts $\dot{u}_{\mathrm{ac}}$ and $\dot{E}_{\mathrm{ac}}$
 of $\dd u$ and $\dd E$
 there holds
\begin{equation}
    \tag{$\ref{s:mainresults}.\Ec_4$}\dot{E}_{\mathrm{ac}}(t_0) \geq \scalx{\xi(t_0),\dot{u}_{\mathrm{ac}}(t_0)}
     + \partial_t \Ec_{t_0}(u(t_0))
    \  \text{  for all } \xi(t_0) \in \diff {t_0}{u(t_0)}\,.
     \label{eq:chainrule}
\end{equation}
\end{assum}
Observe that, since $X$ has the Radon-Nikod\'ym property, the set of
Lebesgue points of $\dot{u}_{\mathrm{ac}}$ and
$\dot{E}_{\mathrm{ac}}$ has full Lebesgue measure in $(0,T)$.

As it
will be clear from the proof of Thm.\ \ref{thm:main-ref} below,
Assumption \ref{ass:chain-rule} does not only provide a motivation
for the energy inequality \eqref{eq-fitzpatrick-BV}, but it also has a key role in the
proof of the passage to the limit as $n \to \infty$ in
\eqref{eq:dne-n}.
 \begin{theorem}\label{thm:main-ref}
  Assume \eqref{X-reflexive}.   Let $\alpha_n,\, \alpha : X \rightrightarrows X^*$ fulfill Assumption \emph{\ref{assum:alpha-ref}},
 and suppose that $\cE: [0,T] \times X \to (-\infty,+\infty]$ complies with Assumptions
 \emph{\ref{ener}} and \emph{\ref{ass:chain-rule}}.
 Let us consider a sequence $(u_0^n) \subset D$ of initial data  such that
 \begin{equation}
 \label{converg-initia-data}
 u_0^n \weakto u_0 \quad \text{in $X$,} \qquad
 \ene 0{u_0^n} \to \ene 0 {u_0},
 \end{equation}
 and let $(u_n,\xi_n) \subset W^{1,1}(0,T;X) \times L^1 (0,T;\Xs)$ be solutions  to
 \eqref{eq:dne-n} in the sense of Definition \ref{def:solutions-AC_case}, fulfilling
  the initial conditions
 $u_n(0) = u_0^n$. Suppose that, in addition, for all
 $n \in \N$ the functions $(u_n,\xi_n)$ comply with the energy identity
 \eqref{eq:defsol}.

  Then, there exist functions $u\in \BV([0,T];X)$ and $\xi\in L^q(0,T;X^*)$ (with
  $q>1$ from \eqref{eq:alpha_0})
  satisfying  $u(0)=u_0$, $\xi(t)\in \diff {t}{u(t)}$ for almost all $t\in (0,T)$, and such that up to a (not relabeled) subsequence
 \begin{equation}
  u_n(t)\rightarrow u(t) \ \forall t\in[0,T]\,, \ \dd u_n=(\dot{u}_n)_\ac\cdot\Ll|_{[0,T]} \wstar \dd u \in \Mm(0,T;X),
 \end{equation}
 and $(u,\xi)$ satisfies  the energy inequality
 \begin{equation}
 \label{eq-fitzpatrick-BV}
 \begin{aligned}
  \Ec_t(u(t)) &  +\int_0^t
   {f_{\alpha}\left(\dot{u}_{\mathrm{ac}}(s),-\xi(s)\right) }\dd s + \int_0^t f_{\alpha}^\infty\left(\dot{u}_\sing(s),0\right)\|(\dd u)_\sing\|(s)
   \\ &  \leq \Ec_0(u(0)) + \int_0^t \partial_t \Ec_s (u(s)) \dd s \qquad \text{for all }\, t \in [0,T],
    \end{aligned}
 \end{equation}
as well as
 \begin{equation}
 \label{interesting-eq}
  \xi(t) \in ({-}\alpha(\dot{u}_\ac (t))) \cap \diff t{ u(t)} \qquad \foraa\, t \in (0,T)\,.
 \end{equation}

Furthermore, there exists $E \in \BV ([0,T])$ such that
\begin{equation}
\label{relation-energies}
E(t)=\ene t{u(t)} \quad \foraa\, t \in (0,T), \qquad
E(t) \geq \ene t{u(t)} \quad \text{for all }t \in [0,T],
\end{equation}
and
we have the \emph{pointwise energy identity}
\begin{equation}
\label{pointiwse-enid}
\dot{E}_\ac (t)+ f_{\alpha}\left(\dot{u}_{\mathrm{ac}}(t),-\xi(t)\right)= \partial_t \ene t{u(t)}
\qquad \foraa\, t \in (0,T).\end{equation}
 \end{theorem}
 \begin{remark}
 \upshape
 In view of Proposition \ref{prop:sol},  a sufficient condition
  for functions $(u_n,\xi_n) $ solving
 \eqref{eq:dne-n} to comply with the energy identity \eqref{eq:defsol}, is that
 they fulfill
 \[
 \sup_{t \in (0,T)} \ene t{u_n(t)}<\infty \qquad \text{and} \qquad \scalx{-\xi_n,{\dot{u}_n}} \in L^1(0,T).
 \]
 This, provided that the \emph{absolutely continuous} version \eqref{ch-rule-AC} of the chain rule holds.

 In Section  \ref{ss:4.3-nr}, we will discuss some sufficient conditions on $\cE$ for both the chain rule \eqref{ch-rule-AC}
 and its $\BV$-version of Assumption \ref{ass:chain-rule} to hold.
 \end{remark}
\subsection{Further results}
\label{ss:4.2-nr}
We conclude this section with some results which
shed light on the interpretation of the energy identities \eqref{eq-fitzpatrick-BV} and
\eqref{pointiwse-enid} satisfied by the pair $(u,\xi)$. More precisely:
\begin{compactenum}
\item[-]
Proposition \ref{lemma:enid}  focuses on the case in which we have the additional information that $u$
is absolutely continuous.  For instance, this is granted whenever $u$ occurs as limiting curve of a sequence $(u_n) \subset
W^{1,1} (0,T:X)$ of solutions to the differential inclusions \eqref{eq:dne-n}, driven by operators $(\alpha_n)$ which
fulfill a stronger version of condition \eqref{eq:alpha_0},  cf.\  Thm.\  \ref{prop:thomas} ahead.
\item[-]
In Proposition \ref{lemma:rate-independent}  we address the special case in which $\alpha=\partial \psi$, with
$\psi: X \to [0,+\infty)$ a convex, lower semicontinuous, and $1$-homogeneous dissipation potential. We show that in this case
any $u \in \BV ([0,T];X)$ complying with the energy inequality \eqref{eq-fitzpatrick-BV} is a \emph{local solution} (cf.\ \cite{Mielke-et-al08, MRS12})
to the rate-independent system $(X,\cE,\psi)$.
\end{compactenum}
\paragraph{\bf The absolutely continuous case.}
 Under a slightly stronger version of the chain rule of Assumption \ref{ass:chain-rule},
 Proposition \ref{lemma:enid}
 shows that,
if in addition we have that the curve $u $ is absolutely continuous on $(0,T)$, then
$f_{\alpha}^\infty\left(\dot{u}_\sing(t),0\right) =0$ for $\|(\dd u)_\sing\|$-a.a.\ $t \in (0,T)$,
and
\eqref{eq-fitzpatrick-BV} holds  on every sub-interval
$[s,t] \subset [0,T]$. Furthermore, the pair
$(u,\xi)$ solves \eqref{eq:dne} in the sense of Definition \ref{def:solutions-AC_case}, cf.\ \eqref{partial-eq}
below.
\begin{proposition}
\label{lemma:enid}
In the framework of \eqref{X-reflexive},
let $ \alpha : X \rightrightarrows X^*$ fulfill \eqref{gen-max-monot},
 and $\cE: [0,T] \times X \to (-\infty,+\infty]$ comply with Assumption
 \emph{\ref{ener}} and with the following chain rule:
 for every
$u\in W^{1,1}(0,T;X)$ and $\xi\in L^1(0,T;X^*)$ such that
\[
\begin{gathered}
    \sup_{t\in[0,T]} \Ec_t (u(t)) < \infty,  \ \ \xi(t)\in \diff t{u(t)} \ \text{for a.a.} \ t\in(0,T)\,,
    \ \
 \int_0^T f_\alpha(\dot u(t),-\xi(t)) \dd t <\infty, 
\end{gathered}
\]
then
\begin{equation}
\label{singzero}
(\dd u)_\sing  =0 \ \ \Rightarrow \ \  (\dd E)_\sing  =0
\end{equation}
 and the chain rule inequality \eqref{eq:chainrule}
 holds.

Let $(u,\xi, E) \in \BV ([0,T];X) \times L^1 (0,T;X^*) \times \BV ([0,T])$ fulfill
  \eqref{relation-energies} and \eqref{pointiwse-enid}.
  Suppose in addition that $u \in W^{1,1}(0,T;X)$.
   Then,
  \begin{equation}
  \label{absol-cont-E}
  E \in W^{1,1}(0,T).
  \end{equation}
 Furthermore, the pair
 $(u,\xi)$ fulfills
 \begin{equation}
 \label{partial-eq}
 -\xi(t) \in \alpha (\dot u (t)) \qquad \foraa\, t \in (0,T),
 \end{equation}
  and there holds the \emph{improved} energy inequality
  \begin{equation}
 \label{eq-fitzpatrick-ineq-better}
 \begin{aligned}
  \Ec_t(u(t))   &  +\int_s^t
   {f_{\alpha}\left(\dot{u}(r),-\xi(r)\right) }\dd r
     \\ & \leq  \Ec_s(u(s)) + \int_s^t \partial_t \Ec_r (u(r)) \dd r  \text{ for all }  t \in (0,T], \ \foraa\, s \in
     (0,t) \text{ and  for } s=0.
    \end{aligned}
 \end{equation}

 Finally, if $\cE$ also fulfills the enhanced chain rule
 \eqref{ch-rule-AC}, then
 \eqref{eq-fitzpatrick-ineq-better} holds as an \emph{equality} for every $0 \leq s \leq t \leq T$.
\end{proposition}
\begin{proof}
Since $u \in W^{1,1}(0,T;X)$, its distributional derivative $\dd u$ has zero singular part, viz.\
$\dd u= \dot{u}_\ac \mathcal{L}$. Then, it follows from \eqref{singzero}
 that
  $\dd E = \dot{E}_\ac \mathcal{L}$, viz.\  $E$ is absolutely continuous.
  Therefore,
  \eqref{pointiwse-enid} becomes
  \begin{equation}
\label{pointiwse-enid-better}
\dot{E} (t)+ f_{\alpha}\left(\dot{u}(t),-\xi(t)\right)= \partial_t \ene t{u(t)}
\qquad \foraa\, t \in (0,T).\end{equation}
Now, combining this with the chain rule inequality \eqref{eq:chainrule}, we  conclude that
$f_{\alpha}\left(\dot{u}(t),-\xi(t)\right)\leq \scalx{-\xi(t),\dot u(t)}$ for almost all $t \in (0,T)$,
hence \eqref{partial-eq} holds.
Then,  to prove \eqref{eq-fitzpatrick-ineq-better}  we integrate  \eqref{pointiwse-enid-better}, thus
obtaining
\begin{equation}
\label{better-enid-E}
E(t)+ \int_s^t f_{\alpha}\left(\dot{u}(r),-\xi(r)\right) \dd r = E(s) +\int_s^t \partial_t \ene r{u(r)} \dd r
\qquad \text{for all } 0 \leq s \leq t \leq T,
\end{equation}
 and we use
\eqref{relation-energies}.

If moreover $\cE$ complies with the chain rule
 \eqref{ch-rule-AC}, then
 $E(t)= \ene t{u(t)}$ for all $t \in [0,T]$,
  since both functions $t \mapsto E(t)$ and $t \mapsto \ene t{u(t)}$ are continuous on
 $[0,T]$ and coincide on a set of full Lebesgue measure. Therefore
 from \eqref{better-enid-E}
  we get \eqref{eq-fitzpatrick-ineq-better}
for $t \mapsto \ene t{u(t)}$. This concludes the proof.
\end{proof}
As  a straightforward consequence of Prop.\ \ref{lemma:enid} we have the following result,
showing that, under a  stronger coercivity
 assumption on the sequence of maximal monotone operators
 $(\alpha_n)$  (cf.\ \eqref{eq:alpha_bis} below), any sequence $(u_n)$ of solutions  to \eqref{eq:dne-n}
 converges up to a subsequence to a curve complying with \eqref{absol-cont-E}--\eqref{eq-fitzpatrick-ineq-better}.
In particular,  observe that,  unlike in  \eqref{eq:alpha_0},
 in  \eqref{eq:alpha_bis} we do not allow the ``degenerate'' value $1$ for  exponent $p$. Indeed, Theorem
 \ref{prop:thomas} below for instance applies to a sequence of operators $\alpha_n = \partial \psi_n$, with
 $\psi_n(v)= 1/{p_n} \| v \|^{p_n}$   and  $p_n \downarrow p>1$ as $n \to \infty$. In this way, we obtain a stability result for
 doubly nonlinear differential inclusions driven by \emph{viscous}
 dissipation potentials, which
   generalizes the results in
  \cite[Thms.\ 3.1, 3.2]{Aizicovici-Yan}.
\begin{theorem}
\label{prop:thomas}
In the frame of \eqref{X-reflexive},
suppose that $\cE: [0,T] \times X \to (-\infty,+\infty]$ complies with
Assumptions \emph{\ref{ener}} and \emph{\ref{ass:chain-rule}}.  Let  $\alpha_n: X\rightrightarrows \Xs$
 fulfill \eqref{gen-max-monot}
 for all $n\in \Nn$ and
    \begin{equation}
    \begin{aligned}
  \exists\, c_1,c_2,c_3>0,  \quad p> 1, \quad  q>1
  \qquad  & \forall\, n \in \N
\quad  \forall  (x,y)\in \alpha_n\,: \\ &  \scalx{y,x} \geq c_1\|x\|^p+ c_2\|y\|_*^q - c_3\,.
 \end{aligned}
 \label{eq:alpha_bis}
    \end{equation}
Suppose that
there exists $\alpha: X\rightrightarrows \Xs$ fulfilling
\eqref{gen-max-monot} such that
$\alpha_n\Graph \alpha$.
 Let  $(u_0^n) \subset D$  be a sequence of initial data fulfilling
\eqref{converg-initia-data}
 and let $(u_n,\xi_n) \subset W^{1,1}(0,T;X) \times L^1 (0,T;\Xs)$ be solutions  to
 \eqref{eq:dne-n}, fulfilling
 $u_n(0) = u_0^n$ and
 \eqref{e:technical-condition} for every $n \in \N$.

Then, there exists $u\in W^{1,p}(0,T;X)$  with $u(0)=u_0$
 such that up to a (not relabeled) subsequence
 \begin{equation}
 \label{compactness}
  u_n(t)\rightarrow u(t) \ \text{ for all } t\in[0,T]\,,  \ \ u_n \weakto u \quad \text{in } W^{1,p}(0,T;X),
 \end{equation}
 and there exists  $\xi\in L^q(0,T;X^*)$
 such that the pair $(u,\xi)$ is a solution to \eqref{eq:dne} in the sense of Definition
 \ref{def:solutions-AC_case}, fulfilling the improved energy inequality \eqref{eq-fitzpatrick-ineq-better}.
\end{theorem}
\noindent
  The \emph{proof} is outlined at the end of Sec.\ \ref{s:proof-thm-1}.

  \paragraph{\bf The rate-independent case.}
  Let us now focus on the case in which
  \begin{equation}
\label{alpha-1-homog}
  \alpha = \partial \psi \text{ with  } \psi : X \to [0,+\infty) \text{ convex, lower semicontinuous and $1$-positively homogeneous}
  \end{equation}
  with associated polar set $K^* \subset X^*$.
 In this case, the energy inequality \eqref{eq-fitzpatrick-BV} rephrases in a more explicit way.
  \begin{proposition}
 \label{lemma:rate-independent}
 Assume \eqref{X-reflexive}.
 Let $\alpha$ fulfill \eqref{alpha-1-homog} and let $(u,\xi) \in \BV (0,T;X) \times L^1 (0,T:X^*)$ satisfy the energy inequality
  \eqref{eq-fitzpatrick-BV}. Then, $(u,\xi)$ fulfill
   \begin{align}
   & \label{loc-stab}
   -\xi(t) \in K^* \quad \foraa\, t \in (0,T),\\
   & \label{rate-indepe-enineq}
   \ene t {u(t)}+  \mathrm{Var}_\psi (u; [0,t]) \leq   \ene 0 {u(0)} +\int_0^t \partial_t \ene s{u(s)} \dd s \quad \text{for all } t \in [0,T]\,,
   \end{align}
   with $ \mathrm{Var}_\psi$ from \eqref{var-psi}.
 \end{proposition}
 \noindent
 In the frame of rate-independent evolution, \eqref{loc-stab} is interpreted as a \emph{local stability condition},
 while the energy inequality \eqref{rate-indepe-enineq} balances the stored energy $\ene t {u(t)}$ and the dissipated energy
 $\mathrm{Var}_\psi (u; [0,t])$, with the initial energy and the work of the external forces $\int_0^t \partial_t \ene s{u(s)} \dd s$.
In fact, the local stability \eqref{loc-stab} and the energy inequality \eqref{rate-indepe-enineq} yield (a slightly weaker version of) the notion of
  \emph{local solution} to the rate-independent system
  $(X,\cE,\psi)$ from
  \cite{Mielke-et-al08, MRS12}.
   Therein, it was observed that this  concept is the weakest among all notions
  of rate-independent evolution, in that it yields the least precise information on the behavior of the solution at jump points.
  On the other hand, \emph{local solutions} arise in the limit of a very broad class of approximations of rate-independent systems.
  This is in the same spirit as the stability results of this work. In particular, notice that the   maximal monotone operators  $\alpha_n$
  converging in the sense of graphs to $\alpha=\partial \psi$ need not be cyclically monotone.  An example in this direction in the plane $X = {\mathbb R}^2$ is given by the graphs $\alpha_n = \partial \psi + (1/n) Q$ where $Q$ is a rotation of $\pi/2$. In this case $\alpha_n \Graph \partial \psi $ but each $\alpha_n$ is noncyclic.
  \medskip

\noindent
We now proceed with the
\begin{proof}[Proof of Proposition \ref{lemma:rate-independent}.]
Let $(u,\xi) \in \BV (0,T;X) \times L^1 (0,T:X^*)$ fulfill
  \eqref{eq-fitzpatrick-BV}.
  Now, in view of  Proposition
  \ref{bella-prop} and of formula \eqref{eqns:repre-psi-b},
  we have
  \begin{equation}
  \label{1st-ingred}
    f_{\alpha}\left(\dot{u}_{\mathrm{ac}}(t),-\xi(t)\right)  = \psi(\dot{u}_{\mathrm{ac}}(t)) + \psi^* (-\xi(t)) = \psi(\dot{u}_{\mathrm{ac}}(t)) + I_{K^*}(-\xi(t))
    \quad
    \foraa\, t \in (0,T)\,.
    \end{equation}
Furthermore,
we have that
 \begin{equation}
  \label{2nd-ingred}
f_{\alpha}^\infty\left(\dot{u}_\sing(t),0\right) \geq \sup_{\xi_0^*
\in K^*} \scalx{\xi_0^*,\dot{u}_\sing(t)} = \psi (\dot{u}_\sing(t) )
\quad \text{for $\|(\dd u)_\sing\|$-a.a. } t \in (0,T)
\end{equation}
where the first inequality is due to \eqref{recession-1-homog} and
the second identity to \eqref{eqns:repre-psi-c}.

Then, taking into account  formula  \eqref{wil-be-used-later} for
$\mathrm{Var}_\psi $, \eqref{eq-fitzpatrick-BV} yields
\[
 \ene t {u(t)}+  \mathrm{Var}_\psi (u; [0,t]) + \int_0^t I_{K^*}(-\xi(s)) \dd s   \leq   \ene 0 {u(0)} +\int_0^t \partial_t \ene s{u(s)} \dd s \quad \text{for all } t \in [0,T],
\]
which is equivalent to \eqref{loc-stab}--\eqref{rate-indepe-enineq}.
\end{proof}


\subsection{Sufficient conditions for closedness and chain rule}
\label{ss:4.3-nr}
Following \cite{MRS11x}, we now show that conditions
of $\lambda$-convexity type on the energy functional $\cE$ ensure
the validity of the closedness property \eqref{cond:E5}, of the
chain rules \eqref{ch-rule-AC}, \eqref{eq:chainrule}, and of
property \eqref{singzero}.

More precisely, in  \cite[Sec.\ 2]{MRS11x} the following \emph{subdifferentiability property} was introduced.
\begin{definition}
\label{def-uniform-subdif}
Let $\cE: [0,T] \times X \to (-\infty,+\infty]$ fulfill \eqref{cond:E0}.
 For every $R>0$,  set
\[
D_{R} = \left\{ u \in D\, : \ \cg(u) \leq R
\right\}\,.
\]
We say
that
$\cE$ is \emph{uniformly subdifferentiable} (w.r.t.\ the variable $u$)
if
 for all $R>0$ there exists a \emph{modulus of
subdifferentiability} $\omega^R:[0,T]\times  D_{R} \times
D_{R}  \to [0,+\infty)$   such that for all $t \in [0,T]$:
\begin{equation}
\label{unform-modulus}
\begin{gathered}
 \text{$\omega_t^R(u,u)=0$ for every $u\in
D_R$,}
\\
  \text{the map $(t,u,v) \mapsto \omega_t^R(u,v)$ is upper
semicontinuous, and}
\\
  \ene tv- \ene tu - \la \xi,v-u\ra \geq
  -\omega^R_t(u,v) \|v-u \| \quad \text{for all $u,\,v \in D_R$
and $\xi \in \diff tu$.}
  \end{gathered}
\end{equation}
\end{definition}
\noindent
It was shown in \cite[Sec.\ 2]{MRS11x} that, a sufficient condition for \eqref{unform-modulus}
is that the map $u \mapsto \ene tu$ is
\emph{$\lambda$-convex}
uniformly in $t \in [0,T]$, namely
\begin{equation}
\label{e:V-convexity}
\begin{aligned}
\exists\, \lambda \in \R \,  \quad &  \forall\, t \in [0,T] \
\forall\,u_0,\, u_1 \in D\ \forall\, \theta \in [0,1]\, : \\
& \ene t{(1-\theta)u_0+\theta u_1}\le (1-\theta)\ene t{u_0}+
  \theta\ene t{u_1}-\frac\lambda2\theta(1-\theta)\|u_0-u_1\|^2\,.
  \end{aligned}
\end{equation}
Suitable  perturbations of $\lambda$-convex functionals also fulfill the closedness and the chain rule properties: we refer to \cite{RS06,MRS11x,RMS08}
for more details and explicit examples.

We have the following
\begin{proposition}
\label{prop:sufficient-conditions}
Let $\cE: [0,T] \times X \to (-\infty,+\infty]$ fulfill \eqref{cond:E0}, \eqref{cond:E2}, and the \emph{uniform subdifferentiability}
condition \eqref{unform-modulus}. Then,
$\cE$
complies with the closedness condition \eqref{cond:E5},  with the chain rules \eqref{ch-rule-AC} and \eqref{eq:chainrule}, and with
 property \eqref{singzero}.
 \end{proposition}
\begin{proof} In \cite[Prop.\ 2.4]{MRS11x}, it was proved that condition \eqref{unform-modulus} implies  \eqref{cond:E5} and \eqref{ch-rule-AC}.
The validity of  \eqref{eq:chainrule} and \eqref{singzero} can be checked trivially adapting the arguments developed for
the proof of \cite[Prop.\ 2.4]{MRS11x}, to which the reader is referred.
\end{proof}

\subsection{Examples of quasistatic limits}
\label{ss:4.3-ulisse} Our approach to the approximation of doubly
nonlinear evolution equations 
 in particular allows us to discuss  quasistatic limits
of dynamical problems. Indeed, the 
 flexibility in the choice of the approximating graphs
$\alpha_n$,  possibly noncyclic monotone, makes it possible to take
rate-independent limits  of  Hamiltonian systems. We
shall provide here some examples of ODEs and PDEs that can be
reformulated within our frame.

Let us start by considering the case of a {\it nonlinearly damped oscillator}. In particular, let $q=q(t)$ represent the set of generalized coordinates of the system, $M$ be the mass matrix, and $U=U(q)$ its smooth and coercive potential energy. Assume moreover that the system dissipates energy in terms of a positively $1$-homogeneous and nondegenerate dissipation potential $D=D(\dot q)$.  By rescaling time $t$ as $  \epsi t  $, the quasistatic limit of the system corresponds to the limit as $\epsi \to 0$ in the equation
\begin{equation}
\epsi^2 M \ddot q + \partial D(\dot q) + \nabla U(q)\ni 0.\label{to_pass_to}
\end{equation}
The latter can be  rephrased  as a single doubly
nonlinear Hamiltonian system  in  the pair $v=(p,q)$, by
 introducing  the Hamiltonian
$H(p,q)=U(p)+q{\cdot}M^{-1}q/2$, the symplectic operator
$$J=
\left( \begin{matrix}
  0 & 1\\-1& 0
\end{matrix}
\right),
$$
and the dissipation potential $\widehat D(\dot p,\dot q)=D(\dot p)$.
 Then, \eqref{to_pass_to} reads
\begin{equation} \partial \widehat D (\dot p,\dot q)  +\epsi J(\dot p,\dot q) +\nabla H(p,q)\ni (0,0),\label{to_pass_to2}
\end{equation}
 which can be  equivalently rewritten as
\begin{align*}
 \partial D(\dot p) + \epsi \dot q + \nabla U(p) &\ni 0,\\
-\epsi \dot p + M^{-1}q&=0,
\end{align*}
Taking the quasistatic limit $\epsi \to 0$ in relation \eqref{to_pass_to2} requires to deal with the graphs
 $\alpha_\epsi= \partial \widehat D + \epsi J$, which are noncyclic monotone for all $\epsi>0$.
 Apart from the coercivity assumption \eqref{eq:alpha_0} (which can however be relaxed in this case),
  this situation fits into our theory. In particular, solution trajectories to the dynamic problem \eqref{to_pass_to} converge to solutions of the corresponding quasistatic limit. By generalizing the choice of the graphs $\alpha_\epsi$, convergence can be obtained for a large class of different approximating problems.

The nonlinear oscillator example can be turned into a first  PDE example by considering the nonlinearly damped semilinear wave equation
\begin{equation}
  \label{semi}
  \epsi^2 u_{tt} + \partial D(u_t) - \Delta u + f(u)=0.
\end{equation}
This is to be posed in the cylinder $\Omega \times (0,T)$ for some
smoothly bounded open set $\Omega \subset \R^n$, along with the
positively $1$-homogeneous and nondegenerate dissipation potential
$D$, the smooth and polynomially bounded function $f$, and suitable
initial and homogeneous Dirichlet boundary conditions (for
simplicity). Equation \eqref{semi} can be variationally reformulated
in terms of a first-order system as
\begin{equation}
  \label{semi2}
   \partial \mathcal D(u_t,v_t) + \epsi \mathcal J(u_t,v_t) + \partial  \mathcal H(u,v)\ni (0,0) \quad \text{in } \mathcal U^* {\times} \mathcal V^*
   \quad \foraa\, t \in (0,T),
\end{equation}
where  $\mathcal U = H^1_0(\Omega)$, $\mathcal V=L^2(\Omega)$, the functionals  $\mathcal D: \mathcal V^2 \to [0,\infty]$, and
$\mathcal H: \mathcal U{\times}\mathcal V \to (-\infty,\infty]$ are given by
$$ \mathcal D(u_t,v_t) =\int_\Omega D(u_t) \dd x, \quad \mathcal H(u,v) =\int_\Omega \left(\frac12 |\nabla u |^2 + \widehat f(u) + \frac12 |v|^2\right) \dd x$$
  for $\widehat f' = f$, and $\mathcal J(u_t,v_t)(x)= J(u_t(x),v_t(x))$ for almost every $x\in \Omega$. Equation \eqref{semi2} fits in our frame along with the choice $\alpha_\epsi = \partial \mathcal D + \epsi \mathcal J$, which are noncyclic monotone for all $\epsi >0$. In particular, owing to our analysis we can take the quasistatic limit $\epsi \to 0$ in the latter (again by suitably circumventing the lack of coercivity, which is inessential here).

Let us now provide a second PDE example by considering the
quasistatic limit in {\it linearized elastoplasticity} with linear
kinematic hardening \cite{hr}. We let $\Omega \subset \R^3$ be the
reference configuration of an elastoplastic body which is subject to
a displacement $u: \Omega \to \R^3$  and a {\it plastic strain} $p:
\Omega \to \Rzd$ (traceless or {\it deviatoric} symmetric
$3{\times}3$ tensors). Then, the evolution of the elastoplastic
medium is described by the system of the (time-rescaled) momentum
balance (in $ \R^3$) and constitutive equation (in $\Rzd$) as
\begin{align*}
  &\epsi^2 \rho u_{tt} - \nabla {\cdot} (\mathbb C(\epsi(u) {-}p))= b, \\
&\partial D(p_t) + \mathbb H p= \mathbb C(\epsi(u) {-}p)
\end{align*}
 in $\Omega \times (0,T)$,  where $\rho=\rho(x) $ stands
for the material density, $\mathbb C$ is the elasticity tensor
(symmetric, positive definite), $\epsi(u) = (\nabla u {+} \nabla
u^\top)/2$ is the symmetrized strain gradient, $b=b(t,x)$ denotes
some body force density, $\mathbb H$ is the hardening tensor, and
$D$ is a positively $1$-homogeneous and nondegenerate dissipation
potential. The choice $D(p_t)=R|p_t|$ for some $R>0$ corresponds to
the classical {\it Von Mises} plasticity. We shall close the latter
elastoplasticity system by imposing homogeneous Dirichlet conditions
on $u$ and no-traction conditions at the boundary (for simplicity).
Then, the system can be recast in the form of a first-order system
by augmenting the variables, including the momentum $v_t=\rho u_t$.
In particular, we can variationally reformulate the system as
\begin{equation}
  \label{semi3}
   \partial \mathcal D(u_t,v_t,p_t) + \epsi \mathcal J(u_t,v_t,p_t) + \partial
    \mathcal H(u,v,p)\ni (b,0,0) \quad \text{in } \mathcal U^* {\times} \mathcal V^*{\times} \mathcal P^*
    \quad \foraa\, t\in (0,T),
\end{equation}
where now the spaces are defined as  $\mathcal U = \{u \in H^1(\Omega;\R^3) :  u=0 \ \text{in} \  \partial \Omega\}$, $\mathcal V=L^2(\Omega;\R^3)$, $\mathcal P = L^2(\Omega;\Rzd)$. The functionals and the operator are given by
\begin{align*}
 \mathcal D(u_t,v_t,p_t) &=\int_\Omega D(p_t)\dd x \quad \forall p_t \in L^1(\Omega;\Rzd),\\
\mathcal H(u,v,p) &=\int_\Omega \left(\frac12  (\epsi(u) {-}p){:}\mathbb C (\epsi(u) {-}p) + \frac12 p{:} \mathbb H p + \frac{1}{2\rho} |v|^2\right) \dd x \quad \forall (u,v,p) \in \mathcal U {\times} \mathcal V{\times} \mathcal P, \\
 \mathcal J(u_t,v_t,p_t) &=
\left(
  \begin{matrix}
 v_t\\ - u_t \\ 0
  \end{matrix}
\right)\quad \forall (u,v,p) \in \mathcal U {\times} \mathcal V{\times} \mathcal P.
\end{align*}
Once again, the operators $\alpha_\epsi = \partial \mathcal D +
\epsi \mathcal J$ are noncyclic monotone for all $\epsi >0$.  In
particular, our analysis is suited in order to analyze the
quasistatic limit $\epsi \to 0$ in the elastoplastic system
\eqref{semi3}. This clearly distinguishes our frame from the former
variational principle from \cite{plas}, which is of no use in the
dynamical case.

\section{\bf Proof of Theorem \ref{thm:main-ref}}
\label{s:proof-thm-1}
\paragraph{\bf Outline.}
Our starting point is the fact that,
 the functions $(u_n,\xi_n)$ fulfill for every $n \in \N$ the energy identity
\begin{equation}
\label{fitz-n}
\Ec_t(u_n(t)) + \int_{0}^t{f_{\alpha_n}\left(\dot{u}_n(s),-\xi_n(s)\right) \dd s} = \Ec_0(u_0^n) +
\int_0^t{\partial_t \Ec_s(u_n(s))\dd s} \qquad \text{for all } t \in (0,T].
\end{equation}
From \eqref{fitz-n}, we will  deduce a priori estimates on the sequence $(u_n,\xi_n)$. Relying on
 well-known strong and weak compactness results,
we will then prove the convergence (up to a subsequence)
 of $(u_n,\xi_n)$ to a limit pair $(u,\hat{\xi})$. Hence we will pass to  the limit as $n \to \infty$ in \eqref{fitz-n}, following the lines of the proof of
  \cite[Thm.\ 4.4]{MRS11x}. Namely,
 we will
 combine  the finite-dimensional lower semicontinuity theorem \cite[Theorem 5.27]{FL07},
  with  tools from infinite-dimensional
  Young measure theory
  (see Appendix \ref{s:appendixB} for some basic recaps),
  and  refined selection arguments mutuated from the proof of
  \cite[Thm.\ 4.4]{MRS11x}. Such arguments  will yield
  the existence of a function  $\xi \in L^1
  (0,T;X^*)$ such that the pair
  $(u,\xi)$ fulfill the energy inequality \eqref{eq-fitzpatrick-BV}.
\begin{notation}
\upshape
\label{notation-for-constants}
Hereafter we will denote by the symbols
  $C, \, C'$ various positive constants,  which may change from line to line,
  only depending on known quantities and in particular
   independent of $n \in \N$. We will also use the place-holders
   \begin{equation}
   \label{place-holders-n}
   E_n(t):= \ene t{u_n(t)}, \qquad P_n(t):= \partial_t \ene t{u_n(t)}.
   \end{equation}
\end{notation}

 \noindent
 \textbf{Step 1 - A priori estimates and compactness:}
 It follows from \eqref{fitz-n}
 and \eqref{cond:E2}
 (cf.\  also estimate \eqref{gronwall-estimate}) that
$
E_n(t) \leq E_n(0) +C_1 \int_0^t E_n(s) \dd s
$ for all $t \in [0,T]$. Since $\sup_{n\in \N} E_n(0) \leq C$ by
\eqref{converg-initia-data}, applying the Gronwall Lemma
   we deduce
   $\sup_{t\in [0,T]}\{E_n(t):t\in[0,T]\} \leq C$.
    Therefore, in view of assumption
    \eqref{cond:E2} and property \eqref{consequence-of-gronwall}, we conclude that
    \begin{equation}
    \label{aprio-est1}
    \exists\, C>0
    \ \ \forall\, n \in \N\, : \quad \sup_{t \in [0,T]}\left( \cg(u_n(t))+ |P_n(t)| \right) \leq C.
    \end{equation}
 Thanks to \eqref{cond:E1} we then infer that
    \begin{equation}
    \label{compact}
    \exists\, K \Subset X \quad \forall\, n \in \N \ \forall\, t \in [0,T]\,: \quad
    u_n(t) \in K.
    \end{equation}
   Then,
   taking into account that $f_{\alpha_n} (\dot{u}_n,-{\xi}_n)  \geq 0$ a.e.\ in $(0,T)$ in view of \eqref{to-be-observed-later},
   \eqref{fitz-n} yields
   \begin{equation}
   \label{aprio-esti3}
 \exists\, C>0 \ \ \forall\, n \in \N\, : \qquad   \|f_{\alpha_n} (\dot{u}_n,-{\xi}_n) \|_{L^1(0,T)} \leq C.
   \end{equation}
    In view of  assumption
    \eqref{eq:alpha_0}, from \eqref{aprio-esti3} we conclude
    \[
    \int_0^T{ c_1\|\dot{u}_n(s)\| + c_2\|{\xi}_n(s)\|_*^q \dd s } \leq C.
    \] Also due to
    \eqref{compact}, we ultimately deduce that
 \begin{equation}
 \label{bounds-u-n}
 \exists\, C>0
    \ \ \forall\, n \in \N\, : \quad
    \| u_n \|_{\BV([0,T];X)} + \|{\xi}_n(s)\|_{L^q(0,T;X^*)} \leq C\,.
 \end{equation}
 Furthermore, from  the energy identity \eqref{fitz-n} we immediately infer that, setting
 $
  h_n(t):= E_n(t) - \int_0^t P_n(s) \dd s$,
  there holds
 \[
 h_n(t)-h_n(s)  =  -\int_{s}^t f_{\alpha_n}(\dot{u}_n(r),-{\xi}_n(r))\dd s \leq 0 \quad \forall\, 0 \leq s \leq t \leq T.
\]
 Therefore we have
 $\Var(h_n; [0,T]) = E_n(0) - E_n(T)) + \int_0^T P_n(s)\dd s \leq C $ thanks to
 \eqref{aprio-est1} and \eqref{converg-initia-data}.
 Since  $(P_n)$ is uniformly bounded in $L^{\infty}(0,T)$,  we conclude that
   \begin{equation}
   \label{aprio-est2}
   \exists\, C>0
    \ \ \forall\, n \in \N\, : \quad
    \Var(E_n; [0,T]) \leq C\,.
   \end{equation}

Estimates \eqref{compact}, \eqref{bounds-u-n}, \eqref{aprio-est2}, and
the Helly principle guarantee that there exists a subsequence $(n_k)$ and
functions
$u\in \BV([0,T];X)$ and  $E\in \BV([0,T])$
such that, as $k \to \infty$,
 \begin{align}
 &
 \label{pointwise-converg}
  \left(u_{n_k}(t),\Ec_t\left(u_{n_k}(t)\right)\right) \rightarrow (u(t),E(t)) \ \ \text{in} \ X \times \Rr \text{ for all $t \in [0,T]$,}
  \\
  &
  \label{measure-converg}
   \dd u_{n_k}= \dot{u}_{n_k} \cdot \Ll \wstar \dd u  \ \ \text{in} \ \Mm(0,T;X).
 \end{align}
 Exploiting Thm.\ \ref{thm:Leb-decomp}, we decompose $\dd u$ as
 \[
 \dd u = (\dd u)_{\mathrm{ac}}+ (\dd u)_\sing=  \dot{u}_{\mathrm{ac}}\, \Ll + \dot{u}_s \, \| (\dd u)_\sing\|.
 \]
 Observe that, by the lower semicontinuity \eqref{cond:E0},
 \begin{equation}
 \label{Egeqe}
 E(t) \geq \ene t{u(t)}
 \qquad \text{for all } t \in [0,T].
 \end{equation}
 Further, in view of estimate
 \eqref{aprio-esti3}, there exists
 $\mu \in \Mm(0,T)$ such that (up to  not relabeled a subsequence)
 \begin{equation}
 \label{falpha-measure-conver}
  f_{\alpha_k}\left(\dot{u}_k(\cdot),-{\xi}_k(\cdot)\right)\cdot \Ll \wstar \mu \ \ \text{in} \ \Mm(0,T)
 \end{equation}
Moreover, by an infinite-dimensional version of the fundamental compactness theorem of Young measure theory
(cf.\ Thm.\ \ref{th:YM1} in Appendix \ref{s:appendixB}),
we can associate with (possibly a subsequence of) $(\xi_{n_k},
P_{n_k})$
 a  limiting Young measure $(\sigma_t)_{t\in(0,T)}\in \mathscr{Y} (0,T;X\times \Rr) $ 
 such that,
 for almost all $t\in(0,T)$
 it holds
  $\sigma_t (X \times \Rr) = 1$ and $\sigma_t$ is supported on the set of the limit points of
  $(\xi_{n_k}(t), P_{n_k}(t))$ w.r.t.\ the  weak topology on $X^*  \times \Rr$, viz.
  \begin{align}
  \label{support-property}
 \text{supp}\left( \sigma_t\right) \subset \bigcap_{j\in\Nn}
 \overline{\left\{\left(\xi_{n_k}(t), P_{n_k}(t)\right): k \geq j \right\}}^{\mathrm{weak}}\,
 \end{align}
 (where with $\overline{B}^{\mathrm{weak}}$ we denote the closure of a set $B \subset X^*  \times \Rr$ w.r.t.\ the weak topology).
  Furthermore, it holds
 \begin{align}
 &
  {\xi}_{n_k}\wconv \int_{X^*\times \Rr} \zeta \dd \sigma_t(\zeta,p) =: {\hat{\xi}} \ \ \ \ \text{in}\ L^q(0,T;X^*) \ \ \text{and}
   \label{weak-converg-xi}
   \\
   &
    P_{n_k} \wstar \int_{X^*\times \Rr} p \dd \sigma_t(\zeta,p) =: \hat{P} \ \ \ \ \text{in} \ L^{\infty}(0,T)\,.
   \label{converg-pi}
 \end{align}

\noindent
 \textbf{Step 2 - Nonemptyness of admissible sets:} From now on, for the sake of simplicity, we
  shall write $k$ instead of $n_k$ .
  There exists a negligible set $N \subset (0,T)$ such that for every $t \in (0,T)
 \setminus N$ convergences \eqref{pointwise-converg} and the support property \eqref{support-property}
 hold.
  Taking into account the closedness condition \eqref{cond:E5},
  it can be easily checked (cf.\ also \cite[Sec.\ 6]{MRS11x}), that
  for almost all $t \in (0,T) $ there holds
  \begin{equation}
  \label{arrayed-properties}
 \begin{array}{lll}
  (t,u(t)) & \in & D(\frsubname),\\
  \Ec_t(u(t)) & = & E(t), \qquad \Ec_0(u(0))= E(0), \\
  \text{supp}\left(\sigma_t\right) & \subset & \{(\zeta,p) \in X^* \times \Rr : \zeta \in \diff {t}{u(t)}\,, \ p\leq \partial_t \Ec_t(u(t))\}\,.
 \end{array}
 \end{equation}
 In particular, from \eqref{converg-pi} and the third of \eqref{arrayed-properties} it follows that
 \begin{equation}
 \label{p-ineq}
 \hat{P}(t) \leq \partial_t \ene t{u(t)} \qquad \foraa\, t \in (0,T)\,.
 \end{equation}

\noindent
 \textbf{Step 3 - $\liminf$ result for the Fitzpatrick function:} 
 In the next lines, we are going to prove that
 \begin{equation}
 \label{crucial-to-observe}
 \begin{aligned}
 &
 \liminf_{k\rightarrow \infty} \int_0^t f_{\alpha_k}\left(\dot{u}_k(r),-{\xi}_k(r)\right)\dd r \\
 & \quad
 \geq \int_0^t \int_{X^*\times \Rr}{ f_{\alpha}\left(\dot{u}_{\mathrm{ac}}(r),-\zeta\right) \dd \sigma_r (\zeta,p)}\dd r + \int_0^t f_{\alpha}^\infty\left(\dot{u}_s(r),0\right) \|(\dd u)_\sing\|(r)\,.
 \end{aligned}
 \end{equation}
 In order to do so,
  employing \cite[Corollary 1.116, p.\  75]{FL07}, we
  decompose the measure $\mu$ from \eqref{falpha-measure-conver} as follows:  there exist
 $\mu_{\mathrm{ac}}$, $\mu_\sing$, $\mu_{\perp}$ in $\Mm(0,T)$ such that
 \begin{equation}
 \label{decomp-mu}
  \begin{array}{c}
  \mu_{\mathrm{ac}} \ll \|(\dd u)_{\mathrm{ac}}\| \,, \ \mu_\sing \ll \|(\dd u)_\sing\|\,,
  \ \mu_{\perp} \perp  \|(\dd u)_{\mathrm{ac}}\| + \|(\dd u)_\sing\| \ \text{and} \\ \mu = \mu_{\mathrm{ac}}+\mu_\sing + \mu_{\perp}
  \end{array}
 \end{equation}
In particular,
 $\mu_{\mathrm{ac}}$ is absolutely continuous w.r.t.\ the Lebesgue measure $\Ll$.
  Since $f_{\alpha_k}(\dot{u}_k, -{\xi}_k)\geq 0$
   a.e.\ in $(0,T)$,
   we obtain $\mu_{\perp}\geq 0$.
   We will split the proof of \eqref{crucial-to-observe} in two  \medskip steps.

   \noindent
   \underline{\emph{First step:}}
   Now, it follows
    from \eqref{decomp-u} and \eqref{decomp-mu}
    and the Radon-Nikod\'ym property of $X$
     that the set of the points $t_0 \in (0,T)$ such that
     $\sigma_{t_0}(X^*\times \Rr) = 1$ and
     \begin{equation}
     \label{full-leb-meas}
 \begin{array}{lll}
 \displaystyle
\dot{u}_{\mathrm{ac}}(t_0) &=&  \displaystyle  \lim_{\veps\rightarrow 0} \frac{\displaystyle(\dd u) \left([t_0-\veps,t_0+\veps]\cap[0,T]\right)}{\displaystyle\veps}\,, \\
  \displaystyle\hat{\xi}(t_0) &=& \displaystyle \lim_{\veps\rightarrow 0} \frac{1}{\veps} \int_{t_0-\veps/2}^{t_0+\veps/2}\hat{\xi}(t)\dd t,\  \text{and}\\
 \displaystyle \frac{\dd \mu_{\mathrm{ac}}}{\dd \Ll}(t_0) &=&  \displaystyle   \lim_{\veps\rightarrow 0} \frac{\displaystyle \mu \left([t_0-\veps,t_0+\veps]\cap[0,T]\right)}{\displaystyle \veps} < \infty
 \end{array}
 \end{equation}
 has full Lebesgue measure.
 From now on,
 we shall use the notation
 \begin{equation}
    \label{q-eps-notation}
    Q_\veps(t_0):=
 [t_0-\veps,t_0+\veps]\cap[0,T]
  \qquad
 \text{with $t_0 \in (0,T)$ such that \eqref{full-leb-meas} holds.}
 \end{equation}
 We then prove that
 \begin{equation}
  \frac{\dd \mu_{\mathrm{ac}}}{\dd\Ll}(t_0) \geq \int_{X^* \times \Rr} f_{\alpha}\left(\dot{u}_{\mathrm{ac}}(t_0),-\zeta\right)\dd\sigma_{t_0}(\zeta,p) \label{eq:liminf-ac}
 \end{equation}
 with $t_0 \in (0,T)$ such that \eqref{full-leb-meas} holds.
 For any such $t_0$, it is also possible to choose a vanishing  sequence $(\veps_m)_m$
 such that for all $m\in \Nn$ there holds
 \begin{equation}
  \mu\left(\{t_0-\veps_m,t_0+\veps_m\}\cap[0,T]\right)=(\dd u)\left(\{t_0-\veps_m,t_0+\veps_m\}\cap[0,T]\right) = 0\,. \label{eq:liminf-nobd}
 \end{equation}
 In order to show \eqref{eq:liminf-ac}, we will  use  \eqref{1dual-fitz-penot}, which yields
 \begin{equation}
 \label{representation-penot}
 f_{\alpha}\left(\dot{u}_{\mathrm{ac}}(t_0),-\zeta\right)= \sup \{\scalx{x^*,\dot{u}_{\mathrm{ac}}(t_0) }
 - \scalx{\zeta,x} - \rho_{\alpha^-1}(x^*,x)\, : \   (x,x^*) \in  X \times X^*
  \}\,.
 \end{equation}
 In view of
 \eqref{representation-penot}, we  thus confine ourselves to showing that
 \begin{equation}
 \label{easier}
 \begin{aligned}
  \frac{\dd \mu_{\mathrm{ac}}}{\dd\Ll}(t_0) &  \geq \scalx{x^*,\dot{u}_{\mathrm{ac}}(t_0)} - \int_{X^* \times \Rr}\scalx{\zeta,x} \dd\sigma_{t_0}(\zeta,p) -\rho_{\alpha^{-1}}(x,x^*)
  \\ & \qquad \qquad \qquad \text{ for all $(x,x^*) \in X \times X^*$ with $\rho_{\alpha^{-1}}(x,x^*)<\infty$.}
  \end{aligned}
 \end{equation}
 Now,  to check
 \eqref{easier}
  we observe that,
  since $\alpha_k\Graph \alpha$ also $\alpha_k^{-1}\Graph \alpha^{-1}$ in the graph sense, we can apply Theorem \ref{Attouch2} to $\rho_{\alpha_k^{-1}} = f_{\alpha_k}^*$.
 Therefore, for any $(x,x^*) \in X\times X^*$ there exists a sequence $\left(x_k, x_k^*\right)_{k\in\Nn}$ such that
 \begin{equation}
 \label{plugged-in}
 \text{
 $(x_k,x_k^*) \rightarrow (x,x^*)$ and $\limsup_{n\rightarrow \infty}\rho_{\alpha_k^{-1}}\left(x_k^*,x_k\right) \leq \rho_{\alpha^{-1}}(x^*,x)$.}
 \end{equation}
Combining \eqref{falpha-measure-conver} with the third of \eqref{full-leb-meas} (for the sequence
$(\eps_m)_m$ fulfilling \eqref{eq:liminf-nobd}), we have that
\begin{equation}
\label{to-quote-later}
 \begin{aligned}
    & \lefteqn{\frac{\dd \mu_{\mathrm{ac}}}{\dd\Ll}(t_0)} \\  & = \lim_{m\rightarrow \infty}{\lim_{k\rightarrow \infty}{\veps_m^{-1}\int_{Q_{\veps_m}(t_0)}{f_{\alpha_k}(\dot{u}_k(t),-{\xi}_k(t))\dd t}}} \\
 &\geq  \liminf_{m\rightarrow \infty}{\liminf_{k\rightarrow \infty}{\veps_m^{-1}\int_{Q_{\veps_m}(t_0)}{\left( \scalx{x_k^*,\dot{u}_k(t)} + \scalx{-{\xi}_k(t),x_k} -\rho_{\alpha_k^{-1}}(x_k^*,x_k) \right)\dd t}}},\\
 \end{aligned}
 \end{equation}
 where in the latter inequality we have
 plugged in the sequence $(x_k,x_k^*)$ from \eqref{plugged-in} and applied formula
 \eqref{representation-penot} for $f_{\alpha_k}$. On account of convergences
 \eqref{measure-converg} and \eqref{weak-converg-xi},
 and of the fact that $(x_k,x_k^*) \rightarrow (x,x^*)$, we have
 for every $m \in \N$
 \begin{align}
 \label{converg1}
 &
  \int_{Q_{\veps_m}(t_0)}{\scalx{x_k^*,\dot{u}_k(t)} \dd t } \rightarrow \int_{Q_{\veps_m}(t_0)}{\scalx{x^*,\dd u(t)}}\,.
\\
&
\label{converg2}
\begin{aligned}
  \int_{Q_{\veps_m}(t_0)}\scalx{{\xi}_k(t),x_k}\dd t \rightarrow  &  \int_{Q_{\veps_m}(t_0)}\scalx{\hat{\xi}(t),x}\dd t
   \\ & =  \int_{Q_{\veps_m}(t_0)} \left( \int_{\Xs \times \R}\scalx{\zeta,x} \dd \sigma_{t} (\zeta,p) \right)\dd t
  \,.
  \end{aligned}
 \end{align}
 Inserting \eqref{converg1}--\eqref{converg2} into \eqref{to-quote-later} and
  using \eqref{plugged-in}, we thus get
  \[
 \begin{aligned} &  {\frac{\dd \mu_{\mathrm{ac}}}{\dd\Ll}(t_0)} \\ &  \geq
  \liminf_{m\rightarrow \infty}  \frac1{\veps_m}
  \left( \int_{Q_{\veps_m}(t_0)}{\scalx{x^*,\dd u(t)}} +  \int_{Q_{\veps_m}(t_0)} \left( \int_{\Xs \times \R}\scalx{-\zeta,x} \dd \sigma_{t} (\zeta,p)- \rho_{\alpha^{-1}}(x,x^*) \right)\dd t \right)
  \end{aligned}
  \]
  and in view of \eqref{full-leb-meas} we infer
   \eqref{easier}, whence \smallskip the desired \eqref{eq:liminf-ac}.

 \noindent
 \underline{\emph{Second step:}}
  choose $t_0\in (0,T)$ such that it satisfies
 \begin{equation}
 \label{arrayed-props-2}
 \begin{array}{lll}
   \displaystyle \dot{u}_\sing(t_0) &= & \displaystyle   \lim_{\veps\rightarrow 0} \frac{ \displaystyle(\dd u) \left([t_0-\veps,t_0+\veps]\cap[0,T]\right)}{\displaystyle \|(\dd u)_\sing
   \|\left([t_0-\veps,t_0+\veps]\cap[0,T]\right)}\,,
   \smallskip
   \\
\displaystyle 0 & = & \displaystyle   \lim_{\veps\rightarrow 0}{\frac{\displaystyle \Ll \left([t_0-\veps,t_0+\veps]\cap[0,T]\right)}{\displaystyle \|(\dd u)_\sing\|\left([t_0-\veps,t_0+\veps]\cap[0,T]\right)}},\ \
 \text{and}
 \smallskip
 \\
 \frac{\displaystyle \dd\mu_\sing}{ \displaystyle \|(\dd u)_\sing\| }(t_0) &= & \displaystyle   \lim_{\veps\rightarrow 0}
  \frac{\displaystyle\mu \left([t_0-\veps,t_0+\veps]\cap[0,T]\right)}{\displaystyle\|(\dd u)_\sing\|\left([t_0-\veps,t_0+\veps]\cap[0,T]\right)} < \infty\,.
 \end{array}
 \end{equation}
 The set of all $t_0$ failing any of
  \eqref{arrayed-props-2}
  is a $\|(\dd u)_\sing\|$-null set.
  We are now going to prove that
 \begin{equation}
 \label{second-claim}
  \frac{\dd\mu_\sing}{ \|(\dd u)_\sing\| }(t_0) \geq f_{\alpha}^{\infty}(\dot{u}_\sing(t_0),0)\,.
 \end{equation}
 for any $t_0 \in (0,T)$ complying with
  \eqref{arrayed-props-2}.
   As before, we will use the notation
  \eqref{q-eps-notation} for the set $Q_\eps(t_0)$ with any such $t_0$, and
   we choose  correspondingly a vanishing sequence $\left(\veps_m\right)$ such that \eqref{eq:liminf-nobd} is satisfied.
   In order to show \eqref{second-claim},
   in view of the representation formula \eqref{repre-recession}  for
   $f_\alpha^\infty$
   it is sufficient to show that
 \begin{equation}
\label{easier-2}
  \frac{\dd\mu_{\sing}}{\|(\dd u)_\sing\| }(t_0) \geq \scalx{x^*,\dot{u}_{\sing}(t_0)} \text{ for all $(x,x^*)$ such that $\rho_{\alpha^{-1}}(x,x^*) < \infty$.}
 \end{equation}
     With the same argument as in the previous lines, we see that
     \begin{equation}
     \label{chain-ineqs}
 \begin{aligned}
 &
 \frac{\dd\mu_{\sing}}{ \|(\dd u)_\sing\| }(t_0) \\  & =  \lim_{m\rightarrow \infty}{\lim_{k\rightarrow\infty} { \frac{1}{\|(\dd u)_\sing\|\left(Q_{\veps_m}(t_0)\right)}\int_{Q_{\veps_m}(t_0)}{f_{\alpha_k}(\dot{u}_k(t), -{\xi}_k(t))\dd t}}} \\
 & \geq \liminf_{m\rightarrow \infty}\liminf_{k\rightarrow\infty}\frac{1}{\|(\dd u)_\sing\|\left(Q_{\veps_m}(t_0)\right)}  \int_{Q_{\veps_m}(t_0)}{\left( \scalx{x_k^*,\dot{u}_k(t)} + \scalx{-{\xi}_k(t),x_k}- \rho_{{\alpha_k}^{-1}}(x_k^*,x_k)
 \right) \dd t}\,,
 \end{aligned}
 \end{equation}
 where $(x_k,x_k^*)$ as in \eqref{plugged-in}  approximates
 $(x,x^*)$ from \eqref{easier-2}.
 Once again,  due to \eqref{measure-converg} and \eqref{weak-converg-xi} we have for every fixed $m \in \N$
  that
 \begin{align*}
  & \int_{Q_{\veps_m}(t_0)}{\scalx{x_k^*,\dot{u}_k(t)}\dd t} \rightarrow \int_{Q_{\veps_m}(t_0)}{\scalx{x^*,\dd u(t)}}\ \text{and}\\
  & \int_{Q_{\veps_m}(t_0)}\scalx{{\xi}_k(t),x_k}\dd t  \rightarrow  \int_{Q_{\veps_m}(t_0)}\scalx{\hat{\xi}(t),x}\dd t\,.
 \end{align*}
 By construction (cf.\ \eqref{arrayed-props-2}), there holds
 \begin{align*}
  & \lim_{m\rightarrow \infty} \frac{1}{\|(\dd u)_\sing\|\left(Q_{\veps_m}(t_0)\right)} \int_{Q_{\veps_m}(t_0)}{\scalx{x^*,\dd u(t)}} = \scalx{x^*, \dot{u}_\sing(t_0)}\,, \\
  & \lim_{m\rightarrow \infty} \frac{1}{\|(\dd u)_\sing\|\left(Q_{\veps_m}(t_0)\right)} \int_{Q_{\veps_m}(t_0)}\scalx{-\hat{\xi}(t),x}\dd t = 0,  \\
  & \lim_{m\rightarrow \infty} \frac{\Ll(Q_{\veps_m}(t_0))}{\|(\dd u)_\sing\|\left(Q_{\veps_m}(t_0)\right)}\rho_{\alpha^{-1}}(x,x^*) = 0\,.
 \end{align*}
 We thus conclude \eqref{easier-2}, whence \eqref{second-claim}.
\medskip

\noindent
In conclusion, passing to the limit as $n_k \to \infty$ in \eqref{fitz-n} and relying on
the initial data convergence \eqref{converg-initia-data}, the energy convergence
\eqref{pointwise-converg} joint with \eqref{Egeqe},
\eqref{converg-pi},
and the lower semicontinuity
\eqref{crucial-to-observe},  we have obtained
 \begin{equation}
   \begin{aligned}
   & \Ec_t(u(t)) +\int_0^t \int_{X^*\times \Rr}{ f_{\alpha}\left(\dot{u}_{\mathrm{ac}}(s),-\zeta\right) \dd \sigma_s (\zeta,p)}\dd s +
    \int_0^t f_{\alpha}^\infty\left(\dot{u}_\sing(s),0\right)\|(\dd u)_\sing\| (s) \\ &
    \leq E(t) +\int_0^t \int_{X^*\times \Rr}{ f_{\alpha}\left(\dot{u}_{\mathrm{ac}}(s),-\zeta\right) \dd \sigma_s (\zeta,p)}\dd s +
    \int_0^t f_{\alpha}^\infty\left(\dot{u}_\sing(s),0\right) \|(\dd u)_\sing\|(s)
    \\ &
      \quad \leq \Ec_0(u(0)) +
     \int_0^t\int_{X^*\times \Rr}p\dd \sigma_s (\zeta,p)\dd s \qquad \text{for all } t \in
     (0,T]\,.
     \end{aligned} \label{eq:liminf-res1}
 \end{equation}

\noindent
 \textbf{Step 4 - Enhanced support properties of  the Young measure $(\sigma_t)_{t \in (0,T)}$:}
We can now improve the third of \eqref{arrayed-properties},
showing that indeed
 \begin{equation}
 \label{enhanced-supp-prop}
   \mathrm{supp}\left(\sigma_t\right)  \subset  \{(\zeta,p) \in X^* \times \Rr : \zeta \in \diff {t}{u(t)}\,,
   \  -\zeta \in \alpha(u(t))\,, \ p\leq \partial_t \Ec_t(u(t))\} \quad \foraa\, t \in (0,T).
 \end{equation}
To this end, observe that, passing to the limit as $n_k \to \infty$
in \eqref{fitz-n} (written on the interval $(s,t)$), yields, in view
of convergences \eqref{pointwise-converg},
\eqref{falpha-measure-conver}, and \eqref{converg-pi}, the following
energy identity
\begin{subequations}
\label{eq:lim-simp}
 \begin{equation}
  E(t) + \mu([s,t]) =E(s) +  \int_s^t p(r) \dd r \qquad \text{for all } 0 \leq s \leq t \leq T \qquad \text{with }\label{eq:lim-simp-a}
 \end{equation}
\begin{equation}
\label{eq:lim-simp-b} \mu([s,t]) \geq \int_s^t \int_{X^* \times \R}
 f_{\alpha}\left(\dot{u}_{\mathrm{ac}}(r),-\zeta\right) \dd \sigma_r
(\zeta,p)\dd r
\end{equation}
the latter inequality due to \eqref{eq:liminf-ac}.
\end{subequations}
In particular, observe that
\begin{equation}
\label{important}
 \int_0^T \int_{X^* \times \R}
 f_{\alpha}\left(\dot{u}_{\mathrm{ac}}(t),-\zeta\right) \dd \sigma_t
(\zeta,p)\dd t<\infty\,.
\end{equation}
  Let $\mathcal{T}\subset[0,T]$ be the set of all
  Lebesgue points
  $t_0$  of $\hat{P}$ \eqref{converg-pi}, such that relations
  \eqref{full-leb-meas} and \eqref{eq:liminf-ac} hold, and
  \[
 \dot{E}_{\mathrm{ac}}(t_0) = \lim_{\veps\rightarrow 0} \frac{E\left(t_0+\frac{\veps}{2}\right)-E\left(t_0-\frac{\veps}{2}\right)}{\veps}\,.
 \]
 Then $\mathcal{T}$ has full measure.
Let us now
 choose
 a sequence $(\veps_m)$, $\veps_m \downarrow 0$, such that \eqref{eq:liminf-nobd} holds.
   Then, passing to the limit as $\eps_m \downarrow
    0$ in
\eqref{eq:lim-simp} (written for $s=t_0 - \eps_m/2$ and $t= t_0 +
\eps_m/2$),  we obtain
\begin{equation}
    \label{clever1}
\dot{E}_{\mathrm{ac}}({t_0})+ \int_{X^*\times \Rr}
f_{\alpha}(\dot{u}_{\mathrm{ac}}(t_0),-\zeta)\dd \sigma_{t_0}(\zeta,p)
\leq
    \int_{X^*\times \Rr} p\dd \sigma_{t_0}(\zeta,p) \quad \text{for all
    } t_0 \in \mathcal{T}
\end{equation}
(up to removing from $\mathcal{T}$ a set of zero Lebesgue measure).
Now, observe that thanks to the third of \eqref{arrayed-properties} and
\eqref{important}, the Young measure $(\sigma_t)_{t \in (0,T)}$
satisfies the assumptions of the forthcoming Lemma \ref{YMchain}. Therefore,
 in view of the Young measure version of the chain rule
inequality \eqref{eq:chainrule} therein, we find that
    \begin{equation}
    \label{clever2}
 \begin{aligned}
  \!\!\!\!\!\!\!\! -\dot{E}_{\mathrm{ac}}({t_0})+ \int_{X^*\times \Rr} p\dd \sigma_{t_0}(\zeta,p)  \leq
  \int_{X^*\times \Rr} \scalx{-\zeta,\dot{u}_{\mathrm{ac}}({t_0})}\dd
   \sigma_{t_0}(\zeta,p)\quad \text{for almost all } t_0 \in
   \mathcal{T}\,.
 \end{aligned}
 \end{equation}
Combining \eqref{clever1} and
\eqref{clever2},
 we deduce that for almost all $t_0 \in \mathcal{T}$ (and
hence for almost all $t_0\in(0,T)$) it holds
 \begin{equation}
\label{very-clever}
   \int_{X^*\times \Rr} \left( f_{\alpha}(\dot{u}_{\mathrm{ac}}(t_0),-\zeta) - \scalx{-\zeta,\dot{u}_{\mathrm{ac}}(t_0)} \right) \dd \sigma_{t_0}(\zeta,p) \leq
   0\,.
 \end{equation}
Since $f_\alpha$ is a representative function for $\alpha$,  we
easily see that \eqref{very-clever} holds as an equality, and that
in fact
\[
-\zeta \in \alpha(\dot{u}_{\mathrm{ac}}(t_0)) \qquad \text{for
$\sigma_{t_0}$-a.a.\ $(\zeta,p) \in \mathrm{supp}(\sigma_{t_0})$.}
\]
Since $t_0  \in (0,T)$ is arbitrary out of a Lebesgue-null set, we
have ultimately proved the desired support property
\eqref{enhanced-supp-prop}. Furthermore, as a by-product of
\eqref{clever1}--\eqref{very-clever} holding as equalities, we infer the
following  pointwise energy equality
\begin{equation}
\label{not-so-surprising}
 \dot{E}_{\mathrm{ac}}({t})+ \int_{X^*\times \Rr}
f_{\alpha}(u_{\mathrm{ac}}(t),-\zeta)\dd
\sigma_{t}(\zeta,p)=\int_{X^*\times \Rr} p\dd \sigma_{t}(\zeta,p)
\qquad \foraa\, t \in (0,T).
\end{equation}

\noindent
 \textbf{Step 5 - Selection argument and conclusion of the proof:}
 We can now apply  Lemma \ref{lem:meas-select}
  and deduce that there exist measurable functions $\xi:(0,T)\rightarrow X$ and
  $P:(0,T)\rightarrow \Rr$ such that
  \begin{equation}
  \label{select-argument}
  \left(\xi(t),p(t)\right) \in \text{\normalfont{argmin}}\left\{f_{\alpha}(\dot{u}_{\mathrm{ac}}(t),-\zeta) -p: (\zeta,p) \in \mathcal{S}\left(t,u(t),\dot{u}_{\mathrm{ac}}(t)\right)\right\} \ \ \text{for a.a.}\ t\in(0,T),
  \end{equation}
  with
    $\mathcal{S}(t,u(t),\dot{u}_{\mathrm{ac}}(t)) := \{(\zeta,p) \in X^*\times \Rr: \zeta \in \diff {t}{u(t)}, \,  -\zeta \in \alpha(\dot{u}_{\mathrm{ac}}(t)),
    \, p\leq \partial_t \Ec_t(u(t))\}$.
    In particular,
      \begin{equation}
      \label{used-later}
  \xi(t)\in \diff t{u(t)}\,, \qquad  -\xi(t) \in \alpha(\dot{u}_{\mathrm{ac}}(t)), \qquad
    \text{and }
     P(t) \leq \partial_t \ene t{u(t)} \quad \foraa\, t \in (0,T),
  \end{equation}
 We then have the following chain of inequalities for almost all $t \in (0,T)$
 \begin{equation}
 \label{ineqs-then-eqs}
 \begin{aligned}
 -\dot{E}_{\mathrm{ac}}({t})  & = \int_{X^*\times \Rr}\left( f_{\alpha}(\dot{u}_{\mathrm{ac}}(t),-\zeta) - p \right)\dd \sigma_t(\zeta,p)
 \\ &
 \geq  f_{\alpha}(\dot{u}_{\mathrm{ac}}(t),-\xi(t)) - P(t)
 \\ & \geq \scalx{-\xi(t), \dot{u}_{\mathrm{ac}}(t) } - \partial_t \ene t{u(t)}
\geq -\dot{E}_{\mathrm{ac}}({t}),
\end{aligned}
 \end{equation}
 where the first identity follows from
 \eqref{not-so-surprising}, the second inequality from
 \eqref{select-argument}, the third one from the fact that $f_\alpha$ is a representative
 function for $\alpha$ and from
 \eqref{used-later}, and the last one from the chain rule inequality \eqref{eq:chainrule}. Therefore
 we infer that all inequalities in \eqref{ineqs-then-eqs} hold as equalities, which proves
 \eqref{pointiwse-enid}.
 In particular,
we have that for almost all $t \in (0,T)$
\begin{equation}
\label{interesting-byprod}
\begin{aligned}
&
P(t)= \int_{X^*\times \Rr} p \dd \sigma_t(\zeta,p)= \partial_t \ene t{u(t)},
\\ &
f_{\alpha}(\dot{u}_{\mathrm{ac}}(t),-\xi(t))= \int_{X^*\times \Rr} f_{\alpha}(\dot{u}_{\mathrm{ac}}(t),-\zeta)
\dd \sigma_t(\zeta,p)= \scalx{-\xi(t),\dot{u}_{\mathrm{ac}}(t)}\,.
\end{aligned}
\end{equation}
 Combining \eqref{interesting-byprod}
 with \eqref{eq:liminf-res1} we ultimately deduce \eqref{eq-fitzpatrick-BV}.
   Now, from \eqref{eq-fitzpatrick-BV} with \eqref{interesting-eq} we have that $\int_0^T |\scalx{-\xi(t),\dot{u}_\ac(t)}| \dd t <\infty$.
 Since $\alpha $ complies with
 \eqref{coerc-alpha},
 we  then conclude that  $\xi \in L^q (0,T;X^*)$.
 This completes the proof.
 \fin
\begin{remark}[The role of the Fitzpatrick function]
\label{rmk:need-fitz}
\upshape
As pointed out in Remark \ref{rmk:towards-ineqs}, the variational reformulation of the doubly nonlinear differential inclusion
\eqref{eq:dne} could be given in terms of \emph{any} representative functional for $\alpha$. The distinguished role of the
Fitzpatrick function $f_\alpha$ is apparent in the passage to the limit argument developed in Step $3$ of the proof of Thm.\ \ref{thm:main-ref}.
Therein (cf.\  \eqref{eq:liminf-ac}--\eqref{converg2}), we exploit the duality formula \eqref{1dual-fitz-penot} for $f_\alpha$, as well as
Theorem \ref{Attouch2}.
\end{remark}
\begin{remark}[Refinement of the measurable selection argument]
\label{rmk:slight-simpl}
\upshape
A close perusal of Step $5$  in the above proof
reveals that, in principle, it
should be sufficient to select the functions $t \mapsto (\xi(t),p(t))$ in the set
  $\tilde{\mathcal{S}}(t,u(t)) := \{(\zeta,p) \in X^*\times \Rr: \zeta \in \diff {t}{u(t)},
    \, p\leq \partial_t \Ec_t(u(t))\}$,
    i.e.\ dropping the requirement
    $-\zeta \in \alpha(\dot{u}_{\mathrm{ac}}(t))$.
    Indeed, if we were in the position of applying Lemma \ref{lem:meas-select} to the set $\tilde{\mathcal{S}}$,
    from the chain of inequalities \eqref{ineqs-then-eqs}
    the second of \eqref{interesting-byprod}  would still
    follow, yielding
    $-\xi(t) \in \alpha(\dot{u}_{\mathrm{ac}}(t))$ for almost all $t \in (0,T)$, i.e.\
 \eqref{interesting-eq}.

Nonetheless, the extension of Lemma \ref{lem:meas-select} to the set $\tilde{\mathcal{S}}$ seems to be an
open problem, at the moment, cf.\ the upcoming Remark \ref{rmk:slight-prob}.
\end{remark}

\noindent
We conclude this section with the
\begin{proof}[Proof of Theorem \ref{prop:thomas}.]
     Repeating the calculations from
     Step $1$ of the proof of Thm.\ \ref{thm:main-ref},
      we prove that
      the sequence $(u_n)$ is bounded in $W^{1,p}(0,T;X)$
      and in addition fulfills estimate \eqref{aprio-est1}.
       Therefore, convergence
       \eqref{compactness} holds. We use the arguments from
       the above
       Steps $1$ and $4$ to infer that there exist $(\xi,E) \in L^1(0,T;\Xs) \times \BV ([0,T])$
       complying with \eqref{relation-energies} and \eqref{pointiwse-enid}. Since $u \in W^{1,p}(0,T;X)$,
       Proposition \ref{lemma:enid} applies, and we conclude the proof.
  \end{proof}

 \begin{appendix}

 \section{\bf Young measure results}
 \label{s:appendixB}
 \noindent We fix here some  definitions and results on
 parameterized (or Young) measures (see   e.g.\
\cite{Balder84, Balder85, Ball89, Valadier90})
with values in  a reflexive Banach space $\uai$.
 In particular,  in Section \ref{s:proof-thm-1} the upcoming results are applied to the space $\uai=
X^* \times \R$.
\begin{notation}
\upshape
\label{not:Borel}
In what follows, we will   denote by $\mathscr{L}_{(0,T)}$
the $\sigma$-algebra of the Lebesgue measurable subsets of $(0,T)$ and
 by $\mathscr B(\uai)$ the  Borel $\sigma$-algebra of $\uai$.
We  use the symbol $\otimes$ for product $\sigma$-algebrae. We
recall that a $\mathscr{L}_{(0,T)} \otimes \mathscr
B(\uai)$-measurable function $h : (0,T) \times \uai \to
(-\infty,+\infty]$ is a \emph{normal integrand} if for a.a. $t \in
(0,T)$ the map $y\mapsto h_t(y)= h(t,y)$ is lower semicontinuous on
$\uai$.

We consider the space $\uai$ endowed   with the \emph{weak}
topology, and say that a $\mathscr{L}_{(0,T)} \otimes \mathscr
B(\uai)$--measurable functional $h: (0,T) \times
\uai \to (-\infty,+\infty]$ is a \emph{weakly-normal
integrand}  if for a.a. $t \in (0,T)$ the map
\begin{equation}
\label{def:wws}
\begin{gathered}
 \text{$y\mapsto h(t,y)$
is sequentially lower semicontinuous on $\uai$ w.r.t.\ the
weak topology.}
\end{gathered}
\end{equation}
We   denote by $\mathscr{M} (0,T; \uai)$ the set of all
$\mathscr{L}_{(0,T)}$-measurable functions $y: (0,T) \to
\uai$. A sequence $(y_n) \subset\mathscr{M} (0,T;
\uai) $ is said to be \emph{weakly-tight}  if there exists a
weakly-normal integrand $h:(0,T) \times \uai
\rightarrow (-\infty,+\infty]  $ such that the map
\[
\begin{gathered}
\text{$y\mapsto  h_t(y)$ has compact sublevels w.r.t.\
 the weak topology of $\uai$, and }
\sup_n \int_0^T h(t,y_n(t))  \dd t <\infty.
\end{gathered}
\]
\end{notation}
 \begin{definition}[\bf Young measures with values in $\uai$]
  \label{parametrized_measures}
  A (time-dependent) \emph{Young measure} in the space $\uai$
  is a family
  ${\bm \sigma}:=\{\sigma_t\}_{t \in (0,T)} $ of Borel probability measures
  on $ \uai$ parameterized by $t \in (0,T)$,
  such that the map on $(0,T)$
\begin{equation}
\label{cond:mea} t \mapsto \sigma_{t}(B) \quad \mbox{is}\quad
{\mathscr{L}_{(0,T)}}\mbox{-measurable} \quad \text{for all } B \in
\mathscr{B}(\uai).
\end{equation}
We denote by $\mathscr{Y}(0,T; \uai)$ the set of all Young
measures in $\uai$.
\end{definition}
\noindent The following result is taken from \cite{MRS11x} (cf.\ Thms.\ A.2 and A.3 therein).
 It is a
 generalization of
the so-called {\em
  Fundamental Theorem of Young measures}
   (cf.\ the
classical results~\cite[Thm.\,1]{Balder84},
\cite[Thm.\,2.2]{Balder85}, \cite{Ball89},
\cite[Thm.\,16]{Valadier90}), to  the case of Young measures with values in $\uai$ endowed
 with the weak topology (see also
\cite[Thm.\,3.2]{RS06}
for the case in which  $\uai$ is a Hilbert space endowed with the weak topology).
 \begin{theorem}[The Fundamental Theorem for weak topologies]
 \label{th:YM1}
  Let
  $\left(y_n\right) \subset \mathscr{M} (0,T; \uai)$
  be
  a weakly-tight sequence. Then,
\begin{enumerate}
   \item
   there exists a subsequence $\left(y_{n_k}\right)$ and a Young measure  ${\bm \sigma} = \left(\sigma_t\right)_{t\in(0,T)} \in \mathscr{Y} (0,T;\uai)$ such that
 \begin{equation} 
  \limsup_{k\uparrow \infty}\|y_{n_k}(t)\|_{\uai} < \infty \quad \text{and}
  \quad \supp(\sigma_t) \subset \bigcap_{j=1}^\infty \overline{\left\{y_{n_k}(t):k\geq j\right\}}^{\mathrm{weak}}
  \quad \foraa\, t \in (0,T),
  \label{eq:YM-concentration}
 \end{equation}
 (where $\overline{B}^{\mathrm{weak}}$ denotes the closure of a set $B \subset \uai$ w.r.t.\ the weak topology),
 and
 such that for every weakly-normal integrand $h:[0,T]\times \uai \rightarrow (-\infty,\infty]$ such that $h^-\left(\cdot,y_{n_k}(\cdot)\right)$ is uniformly integrable it holds
 \begin{equation}
  \liminf_{k\rightarrow \infty}\int_0^T{ h\left(t,y_{n_k}(t)\right) \dd t} \geq\int_0^T{ \int_{\uai}{h\left(t,y\right)\dd\sigma_t(y)} \dd t}\,.\label{eq:YM-liminf}
 \end{equation}
\item In particular, let
$(y_n) \subset L^q
(0,T;\uai)$ be a bounded sequence, with
 $q\in (1,+\infty]$.
   Then,
  there exists a further (not relabeled) subsequence $(y_{n_k})$ and
  a Young  measure
  $ {\bm \sigma}=\{ \sigma_{t} \}_{t \in (0,T)}  \in \mathscr{Y}(0,T;\uai)$
  such that for a.a.\ $t\in (0,T)$ properties  \eqref{eq:YM-concentration} hold.
  Setting
  $
\mathrm{y}(t):=\int_{\uai} y \, \dd \sigma_t (y) $ for almost all $t
\in (0,T)$,
there holds
\begin{equation}
  \label{eq:35}
y_{n_k} \weaksto \mathrm{y} \ \ \text{ in $L^p (0,T;\uai)$}.
\end{equation}
\end{enumerate}
 \end{theorem}
 \subsection{A Young-measure version of the chain rule}
 In what follows, we will work with Young measures with values in the space
 $\uai= X^*\times \R$.
Our first result,
  a small variation  of \cite[Prop.\ B.1]{MRS11x}, provides the version of the chain rule inequality
  \eqref{eq:chainrule} in terms of Young measures used in Step $4$ of the proof of Thm.\
  \ref{thm:main-ref}.
 \begin{lemma}\label{YMchain}
 In the frame of \eqref{X-reflexive},   let $\alpha: X \rightrightarrows \Xs$ fulfill \eqref{gen-max-monot} and the coercivity condition
    \eqref{coerc-alpha}, and let $\cE: [0,T] \times X \to (-\infty,+\infty]$ comply with
    Assumption \emph{\ref{ass:chain-rule}}.
    Let
     $u\in \BV([0,T];X)$ satisfy
    \begin{equation}
    \begin{gathered}
        \sup_{t\in [0,T]} \Ec_t(u(t)) < \infty\,, \ \ \left(t,u(t)\right) \in \dom(\frsubname) \ \text{for a.a.} \ t\in (0,T), \ \ \int_0^T |\partial_t \ene t{u(t)}|\dd t <\infty,
        \\
        \exists E \in \BV ([0,T]) \text{ such that } E(t) = \ene t{u(t)} \ \foraa\, t \in (0,T),
         \label{app:assum-state}
         \end{gathered}
    \end{equation}
    and let $(\sigma_t)_{t\in(0,T)}\in \mathscr{Y}(0,T;X^* \times \R)$
    be a  Young measure such that
    \begin{align}
        & \forall (\xi,p) \in \supp(\sigma_t): \xi \in \diff {t}{u(t)}\,, \ p \leq \partial_t \Ec_t(u(t)) \text{ for a.a.} \ t\in (0,T)
        \label{s:app-conce}
         \\
        & \int_0^T \int _{X^* \times \Rr}{f_{\alpha}(\dot{u}_{\mathrm{ac}}(s),-\zeta)\dd \sigma(\zeta,p)}ds < \infty
        \label{s:app-falpha}
    \end{align}
    Then, for almost all $t\in (0,T)$ such
    that $t$ is Lebesgue point of $\dot{E}_{\mathrm{ac}}$
    and $\dot{u}_{\mathrm{ac}}$  there holds
    \begin{equation}
        \dot{E}_{\mathrm{ac}}(t) \geq \int_{X^*\times \Rr}\left(\scalx{\zeta,\dot{u}_{\mathrm{ac}}(t)} + p \right)\dd \sigma_t(\zeta,p)\,. \label{eq:YMchain}
    \end{equation}
 \end{lemma}
 \begin{proof}
     We consider the set $K(t,u(t)) := \{(\xi,p) \in X^*\times \Rr\, : \  \xi\in \diff t{u(t)}, \, p \leq \partial_t \Ec_t(u(t))\} $.
      Repeating the very same arguments as in the proof of \cite[Prop.\ B.1]{MRS11x},
      we can show that
     there exists a sequence $(\xi_n,p_n)$ of strongly measurable functions
      $(\xi_n,p_n): (0,T) \to X^* \times \R$
      such that
    \begin{equation}
        \left\{(\xi_n(t),p_n(t))\, : \  n \in \N\right\} \subset K(t,u(t)) \subset \overline{\left\{(\xi_n(t),p_n(t))\, : \  n \in \N\ \right\} }\quad \foraa\, t \in (0,T) \label{eq:dense}
    \end{equation}
    (where $\overline{B}$ denotes the closure of $B \subset \Xs\times \R$
w.r.t. the strong topology of $\Xs\times \R$).

  We now claim
  that
    the sequence $(\xi_n,p_n)$ can be chosen such that
    \begin{equation}
     \forall n\in \Nn: \xi_n \in L^1(0,T;X^*) \ \ \text{and} \ \sup_{n\in\Nn} \int_0^T{f_{\alpha}(\dot{u}_\mathrm{ac}(t),-\xi_n(t))\dd t} < \infty\,. \label{app:boundedness}
    \end{equation}
    To this aim, we define the function
    $
     g(t):= \inf\{f_{\alpha}(\dot{u}_{\mathrm{ac}}(t),-\zeta): (\zeta,p)\in K(t,u(t))\} $  for almost all  $t\in(0,T)$.
    Notice that due to \eqref{eq:dense} it holds
    \begin{equation}
    \label{g-measurability}
     g(t) := \inf_{n\in\Nn} \{f_{\alpha}(\dot{u}_{\mathrm{ac}}(t),-\xi_n(t))\} \ \text{for a.a.} \ t\in(0,T)
    \end{equation}
    and hence $g$ is measurable on $(0,T)$. Moreover,
    \begin{equation}
    \label{f-alpha-bound}
     \int_0^T g(t)\dd t \leq \int_0^T \int_{X\times \Rr} f_{\alpha}(\dot{u}_{\mathrm{ac}}(t),-\zeta)\dd \sigma_t(\zeta,p) \dd t < \infty.
    \end{equation}
     With a straightforward adaptation of
     the argument of \cite[Prop.\ B.1]{MRS11x} (see also \cite[Lemma 3.4]{RS06}),
      from \eqref{g-measurability} and
      \eqref{f-alpha-bound}
      we deduce
       \eqref{app:boundedness}.

        In view of the  obtained \eqref{eq:dense} and \eqref{app:boundedness},
         we are
        in the position to apply the chain rule inequality \eqref{eq:chainrule}
        to the pair $(u,\xi_n)$ for every $n \in \N$.
        Therefore for every $n \in \N$ there exists a set
         $\mathcal{T}_n \subset (0,T)$ of full measure such that
    $
     \dot{E}_\mathrm{ac}(t)\geq \scalx{\xi_n(t),\dot{u}_{\mathrm{ac}}(t)} + p_n(t) $
     for all $t \in \mathcal{T}_n$, where we have also used that
     $p_n(t) \leq \partial_t \ene t{u(t)}$.
    The set $\mathcal{T} = \bigcap_{n\in\Nn} \mathcal{T}_n$,
 has still full measure, and there holds for all $t \in \mathcal{T}$
    \begin{equation}
    \label{convex-hull}
     \dot{E}_{\mathrm{ac}}(t)\geq \scalx{\zeta,\dot{u}_{\mathrm{ac}}(t)} + p \quad \text{for all }
     (\zeta, p) \in \overline{\text{conv}\ K(t,u(t))},
    \end{equation}
    the latter set denoting the closed convex hull of $ K(t,u(t))$. Integrating
    \eqref{convex-hull} w.r.t.\ the measure $\sigma_t$
   we obtain \eqref{eq:YMchain}.
 \end{proof}
We conclude with the measurable selection
result exploited in Step $5$ of the proof of Thm.\ \ref{thm:main-ref}.
 \begin{lemma}\label{lem:meas-select}
 In the framework of \eqref{X-reflexive},
 let $\alpha: X \rightrightarrows \Xs$ fulfill \eqref{gen-max-monot} and the coercivity condition
    \eqref{coerc-alpha}, and let $\cE: [0,T] \times X \to (-\infty,+\infty]$ comply with
    Assumptions \emph{\ref{ener}} and   \emph{\ref{ass:chain-rule}}.
 Furthermore, let $u\in \BV([0,T];X)$  fulfill  \eqref{app:assum-state}. Suppose that
 for almost all $t\in(0,T)$
 \begin{equation}
 \label{def-subsetS}
  \mathcal{S}(t,u(t),\dot{u}_{\mathrm{ac}}(t)) := \left\{(\zeta,p) \in X^*\times \Rr: \zeta \in \diff {t}{u(t)}\,,\ -\zeta \in \alpha(\dot{u}_{\mathrm{ac}}(t))\,,
   p\leq \partial_t \Ec_t(u(t))\right\} \neq \emptyset\,.
 \end{equation}
Then, there exist measurable functions $\xi:(0,T)\rightarrow X^*$ and $P:(0,T)\rightarrow \Rr$ such that
 \begin{equation}
  \left(\xi(t),P(t)\right) \in \text{\normalfont{argmin}}\left\{f_{\alpha}(\dot{u}_{\mathrm{ac}}(t),-\zeta) -p: (\zeta,p) \in \mathcal{S}(t,u(t),\dot{u}_{\mathrm{ac}}(t))\right\} \ \ \text{for a.a.}\ t\in(0,T)\,. \label{app:sel-prop}
 \end{equation}
 \end{lemma}
\begin{proof}
The argument follows  the very same  lines of  \cite[Lemma B.2]{MRS11x}.
First of all, we
observe that
\begin{equation}
\label{non-empty-argmin}
\text{\normalfont{argmin}}\left\{f_{\alpha}(\dot{u}_{\mathrm{ac}}(t),-\zeta) -p: (\zeta,p) \in \mathcal{S}(t,u(t),\dot{u}_{\mathrm{ac}}(t))\right\}\neq \emptyset \ \ \text{for a.a.}\ t\in(0,T)\,.
\end{equation}
To this aim, let $(\zeta_n,p_n) \subset \mathcal{S}(t,u(t),\dot{u}_{\mathrm{ac}}(t))$ be an infimizing sequence:
then there exist constants $C, \, C' >0$ such that for every $n \in \N$
\begin{equation}
\label{hence-coercivity}
C \geq
f_{\alpha}(\dot{u}_{\mathrm{ac}}(t),-\zeta_n)-p_n= \scalx{-\zeta_n,\dot{u}_{\mathrm{ac}}(t)}
-p_n \geq  c_1\|\dot{u}_{\mathrm{ac}}(t)\|^p+ c_2\|\zeta_n\|_*^q - c_3 - C'
 \end{equation}
 where  we have used that $-\zeta_n \in \dot{u}_{\mathrm{ac}}(t)$, the coercivity property
 \eqref{coerc-alpha} of $\alpha$, and  that $p_n \leq \partial_t \Ec_t(u(t)) \leq C$
 due to the fact that $\sup_{t \in [0,T]} \ene t{u(t)}<\infty$ and to \eqref{cond:E2}.
Therefore, we infer that
$
\sup_{n \in \N} ( \|\zeta_n\|_*^q + |p_n|) <\infty.
$
Hence, there exist $(\zeta,p) \in \Xs \times \R$
such that, up to a not relabeled subsequence,
$\zeta_n\weakto \zeta$ in $X^*$ and $p_n \to p$. Thanks to the closedness condition \eqref{cond:E5}
and to the weak closedness of $\alpha(\dot{u}_{\mathrm{ac}}(t))$,
we have $(\zeta,p)  \in \mathcal{S}(t,u(t),\dot{u}_{\mathrm{ac}}(t)) $.
Using that $\zeta \mapsto f_\alpha (\dot{u}_{\mathrm{ac}}(t),-\zeta)$
is (sequentially) weakly-lower semicontinuous, we conclude that
\[
\liminf_{n \to \infty} \left( f_{\alpha}(\dot{u}_{\mathrm{ac}}(t),-\zeta_n) -p_n \right)
\geq f_{\alpha}(\dot{u}_{\mathrm{ac}}(t),-\zeta) -p
\]
and \eqref{non-empty-argmin} ensues.

Once obtained \eqref{non-empty-argmin}, the argument for \eqref{app:sel-prop} is a straightforward
adaptation of the proof of \cite[Lemma B.2]{MRS11x}, to which we refer for
all details. Let us only mention here that the existence of
  $(\xi,P)$ as in \eqref{app:sel-prop}  is a consequence of the  measurable selection results  \cite[Cor.\ III.3, Thm.\ III.6]{Castaing-Valadier77}.
\end{proof}

\begin{remark}
\label{rmk:slight-prob}
\upshape
Let us stress that the requirement
$\zeta \in \alpha(\dot{u}_{\mathrm{ac}}(t))$  in the definition \eqref{def-subsetS} of the set
$\mathcal{S}(t,u(t),\dot{u}_{\mathrm{ac}}(t))$ has a crucial role in proving that
$$
\text{\normalfont{argmin}}\{f_{\alpha}(\dot{u}_{\mathrm{ac}}(t),-\zeta) -p: (\zeta,p) \in S(t,u(t),\dot{u}_{\mathrm{ac}}(t))\}$$ is nonempty. In fact, it ensures the estimates in
\eqref{hence-coercivity} for any infimizing sequence
$(\zeta_n,p_n)$.
\end{remark}
 \end{appendix}

\bibliographystyle{plain}

\def\cprime{$'$} \def\cprime{$'$}

\end{document}